\documentclass[twoside,a4paper,reqno,11pt]{amsart}
\usepackage{amsfonts, amsbsy, amsmath, amssymb, latexsym}
\usepackage{mathrsfs,array,enumitem}
\usepackage[top=30mm,right=30mm,bottom=30mm,left=30mm]{geometry}
\usepackage{stmaryrd}
\usepackage{bm}

\renewcommand{\a}{\alpha}

\newcommand{\normeq}{\trianglelefteqslant}

 \renewcommand{\L}{\Lambda}
\renewcommand{\l}{\lambda} \renewcommand{\O}{\Omega}

 \renewcommand{\to}{\rightarrow}
 \newcommand{\s}{\sigma}

\newcommand{\la}{\langle}
\newcommand{\ra}{\rangle}

\newcommand{\leqs}{\leqslant}
\newcommand{\geqs}{\geqslant}

 \newcommand{\vs}{\vspace{3mm}}

\def\GL{{\rm GL}}

\makeatletter
\newcommand{\imod}[1]{\allowbreak\mkern4mu({\operator@font mod}\,\,#1)}
\makeatother

\newtheorem{theorem}{Theorem}
\newtheorem*{conj*}{Conjecture}

\newtheorem{thm}{Theorem}[section]
\newtheorem{prop}[thm]{Proposition}
\newtheorem{lem}[thm]{Lemma}
\newtheorem{cor}[thm]{Corollary}

\theoremstyle{definition}
\newtheorem{rem}[thm]{Remark}
\newtheorem{example}[thm]{Example}
\newtheorem{remk}{Remark}

\newtheorem*{def-non}{Definition}
\newtheorem*{rem-non}{Remark}

\def\PSp{{\rm PSp}}

\def\GammaL{{\rm \Gamma L}}
\def\AGammaL{{\rm A\Gamma L}}

\def\GL{{\rm GL}}
\def\AGL{{\rm AGL}}

\begin{document}

\author{Timothy C. Burness}
 \address{T.C. Burness, School of Mathematics, University of Bristol, Bristol BS8 1UG, UK}
 \email{t.burness@bristol.ac.uk}

\author{Cai Heng Li}
\address{C.H. Li, SUSTech International Center for Mathematics and Department of Mathematics, Southern University of Science and Technology, Shenzhen 518055, Guangdong, P.R. China}
\email{lich@sustech.edu.cn}

\title[On solvable factors of almost simple groups]{On solvable factors of almost simple groups}

\begin{abstract}
Let $G$ be a finite almost simple group with socle $G_0$. A (nontrivial) factorization of $G$ is an expression of the form $G=HK$, where the factors $H$ and $K$ are core-free subgroups. There is an extensive literature on factorizations of almost simple groups, with important applications in permutation group theory and algebraic graph theory. In a recent paper, Li and Xia describe the factorizations of almost simple groups with a solvable factor $H$. Several infinite families arise in the context of classical groups and in each case a solvable subgroup of $G_0$ containing $H \cap G_0$ is identified. Building on this earlier work, in this paper we compute a sharp lower bound on the order of a solvable factor of every almost simple group and we determine the exact factorizations with a solvable factor. As an application, we describe the finite primitive permutation groups with a nilpotent regular subgroup, extending classical results of Burnside and Schur on cyclic regular subgroups, and more recent work of Li in the abelian case. 
\end{abstract}

\date{\today}
\maketitle

\section{Introduction}\label{s:intro}

A group $G$ is said to be \emph{factorizable} if there exist proper subgroups $H$ and $K$ such that $G=HK$. Such an expression is called a \emph{factorization} of $G$ and $H,K$ are called its \emph{factors}. We say that $G=HK$ is a \emph{maximal} factorization if both factors are maximal subgroups of $G$, and it is an \emph{exact} factorization if $H \cap K=1$. Factorizations arise naturally in permutation group theory; if $G \leqs {\rm Sym}(\O)$ is a transitive permutation group with point stabilizer $K$, then a subgroup $H \leqs G$ is transitive on $\O$ if and only if $G=HK$. Moreover, $H$ is regular if and only if the factorization is exact.

In this paper, we are interested in factorizations of almost simple groups (recall that a finite group $G$ is \emph{almost simple} if $G_0 \leqs G \leqs {\rm Aut}(G_0)$ for a nonabelian simple group $G_0$). In this setting, we say that a factorization $G=HK$ is \emph{nontrivial} if both factors are core-free subgroups of $G$ (that is, neither $H$ nor $K$ contains the socle $G_0$). The problem of determining the (nontrivial) factorizations of almost simple groups has a long history and numerous applications. For example, It\^{o} \cite{ito} determined the factorizations of ${\rm PSL}_{2}(q)$, and the exact factorizations of alternating and symmetric groups were classified by Wiegold and Williams \cite{WW}, and later Liebeck, Praeger and Saxl described the general factorizations of these groups (see \cite[Theorem D]{LPS}). The factorizations of almost simple groups with socle an exceptional group of Lie type are determined in \cite{HLS}.

The maximal factorizations of all almost simple groups were classified by Liebeck, Praeger and Saxl \cite{LPS} and this landmark paper from 1990 has been the foundation for more recent progress on the general problem. For example, Guralnick and Saxl \cite[Theorem 3.1]{GS} apply \cite{LPS} to describe the factorizations $G=HK$ of an almost simple group such that $H/O_p(H)$ is cyclic for some prime $p$, and they use this to study the structure of monodromy groups of polynomials. In a different direction, Giudici \cite{G} works with \cite{LPS} to identify  the factorizations of almost simple groups with socle a sporadic simple group. Much more recently, Li and Xia \cite{LX0} have used \cite{LPS} to determine the factorizations of almost simple groups with the property that one of the factors has at least two non-solvable composition factors.

Factorizations of almost simple groups have many natural applications, especially when combined with the O'Nan-Scott Theorem, which describes the structure and action of finite primitive permutation groups. For example, see \cite{Bau1, Li-Abel-gps,LPS-3,LPS-4} for related work on primitive groups containing transitive subgroups with prescribed properties, and we refer the reader to \cite{Li-Lu-Pan,LPS-2} for applications to the study of Cayley graphs.

Let $G$ be an almost simple group with socle $G_0$. The general problem of determining the factorizations $G=HK$ with $H$ solvable is investigated in depth by Li and Xia in \cite{LX} (as noted above, a special case of this problem was studied in earlier work by Guralnick and Saxl \cite[Theorem 3.1]{GS}). Their main result, \cite[Theorem 1.1]{LX}, describes the possibilities for $G$, $H$ and $K$, and the proof treats each family of almost simple groups in turn. By \cite{HLS}, there are no such factorizations if $G_0$ is an exceptional group of Lie type, and it is straightforward to read off the cases that arise when $G_0$ is a sporadic simple group via \cite{G}. For alternating and symmetric groups, the analysis combines \cite{LPS} with earlier work of Kantor \cite{Kantor} on $k$-homogeneous groups. In particular, if $G_0$ is non-classical, then Li and Xia give a detailed description of the triples $(G,H,K)$ that arise.

Determining the factorizations of an almost simple classical group with a solvable factor is more difficult. Here \cite[Theorem 1.1]{LX} states that if $G$ is such a group with socle $G_0$, then $G=HK$ with $H$ solvable only if one of the following holds:
\begin{itemize}\addtolength{\itemsep}{0.2\baselineskip}
\item[{\rm I.}] $(G_0, H \cap G_0, K\cap G_0)$ belongs to one of 9 infinite families listed in \cite[Table 1.1]{LX};
\item[{\rm II.}] $(G_0, H \cap G_0, K\cap G_0)$ is one of 28 ``sporadic" cases recorded in \cite[Table 1.2]{LX};
\item[{\rm III.}] Both $H$ and $K$ are solvable and $(G,H,K)$ is described in \cite[Proposition 4.1]{LX} (using earlier work of Kazarin \cite{Kaz}); here one infinite family arises (with $G_0 = {\rm PSL}_{2}(q)$), together with 11 ``sporadic" cases listed in \cite[Table 4.1]{LX}.
\end{itemize}

In order to describe the infinite families in \cite[Table 1.1]{LX}, Li and Xia identify specific subgroups $A$ and $B$ of $G_0$ (with $A$ solvable) such that
\[
H \cap G_0 \leqs A, \;\; K \cap G_0 \trianglerighteqslant B,
\]
and in every case they demonstrate the existence of an almost simple group $G$ with socle $G_0$ and a factorization $G=HK$ satisfying the given conditions. For the reader's convenience, the infinite families are listed in Table \ref{tab}, together with the subgroups $A$ and $B$ (here $P_k$ denotes the stabilizer of a $k$-dimensional totally singular subspace of the natural module for $G_0$).
Similarly, a solvable subgroup $A$ of $G_0$ with $H \cap G_0 \leqs A$ is recorded for each of the factorizations in \cite[Table 1.2]{LX}, together with the structure of $K \cap G_0$ (as before, in every case there is a genuine factorization of an almost simple group $G$ satisfying the given conditions). In view of the above, we will refer to Type I, II and III factorizations of an almost simple classical group, accordingly.

\begin{table}
\[
\begin{array}{cllll} \hline
\mbox{Case} & G_0 & A & B & \hspace{3.7mm} \mbox{Conditions} \\ \hline
1 & {\rm PSL}_{n}(q) & \frac{q^n-1}{(q-1)d}{:}n & q^{n-1}{:}{\rm SL}_{n-1}(q) & \hspace{3.7mm} d = (n,q-1) \\
2 & {\rm PSL}_{4}(q) & q^{3}{:}\frac{q^3-1}{d}{:}3 < P_k & {\rm PSp}_{4}(q) & \hspace{-1mm} \left\{\begin{array}{l}
d = (4,q-1) \\ k \in \{1,3\} \end{array}\right. \\
3 & {\rm PSp}_{2m}(q) & q^{\frac{1}{2}m(m+1)}{:}(q^m-1).m<P_m & \O_{2m}^{-}(q) & \hspace{3.7mm} \mbox{$m \geqs 2$, $q$ even} \\
4 & {\rm PSp}_{4}(q) & q^{3}{:}(q^2-1).2<P_1 & {\rm Sp}_{2}(q^2) & \hspace{3.7mm} \mbox{$q$ even} \\
5 & {\rm PSp}_{4}(q) & q^{1+2}{:}\frac{q^2-1}{2}.2<P_1 & {\rm PSp}_{2}(q^2) & \hspace{3.7mm}\mbox{$q$ odd} \\
6 & {\rm PSU}_{2m}(q) & q^{m^2}{:}\frac{q^{2m}-1}{(q+1)d}.m<P_m & {\rm SU}_{2m-1}(q) & \hspace{-1mm} \left\{\begin{array}{l} m \geqs 2 \\ d = (2m,q+1) \end{array}\right. \\
7 & \O_{2m+1}(q) & (q^{\frac{1}{2}m(m-1)}.q^m){:}\frac{q^m-1}{2}.m<P_m & \O_{2m}^{-}(q) & \hspace{3.7mm} \mbox{$m \geqs 3$, $q$ odd} \\
8 & {\rm P\O}_{2m}^{+}(q) & q^{\frac{1}{2}m(m-1)}{:}\frac{q^m-1}{d}.m<P_k & \O_{2m-1}(q) & \hspace{-1mm}\left\{\begin{array}{l} m \geqs 5 \\ d = (4,q^m-1) \\ k \in \{m,m-1\} \end{array}\right. \\
9 & {\rm P\O}_{8}^{+}(q) & q^6{:}\frac{q^4-1}{d}.4<P_k & \O_{7}(q) & \hspace{-1mm} \left\{\begin{array}{l} d = (4,q^4-1) \\ k \in \{1,3,4\} \end{array}\right. \\ \hline
\end{array}
\]
\caption{The Type I factorizations of classical groups in \cite[Table 1.1]{LX}}
\label{tab}
\end{table}

With applications in mind, it is desirable to extend \cite[Theorem 1.1]{LX} by obtaining more information on the Type I factorizations of almost simple classical groups. In particular, we would like to shed more light on the minimal possibilities for a solvable factor in these cases (cf. \cite[Problem 1.7]{LX}).

Our first result gives a sharp lower bound on the order of a solvable factor in every Type I factorization. In the statement, a \emph{Type I.j factorization} refers to a factorization satisfying the conditions of Case j in Table \ref{tab}.

\begin{theorem}\label{t:main}
Let $G$ be an almost simple classical group with socle $G_0$ and suppose $G=HK$ is a Type I.j factorization of $G$. Then there exists an integer $\ell$ and an almost simple group $G_1$ with socle $G_0$ such that
\begin{itemize}\addtolength{\itemsep}{0.2\baselineskip}
\item[{\rm (i)}] $G_1 = H_1K_1$ is a Type I.j factorization of $G_1$; and
\item[{\rm (ii)}] $|H_1| = \ell$ divides $|H|$,
\end{itemize}
where $(G_1,H_1,K_1,\ell)$ is listed in Table $\ref{tab2}$.
\end{theorem}

\begin{remk}\label{r:1}
In Table \ref{tab2}, we use the same numbering of cases as in Table \ref{tab} and we assume all of the conditions stated in the final column of that table. In addition, we adopt the following conventions in Table \ref{tab2}:
\begin{itemize}\addtolength{\itemsep}{0.2\baselineskip}
\item[{\rm (a)}] It is convenient to allow $n=2$ in Case 1, which is strictly speaking a Type III factorization as the parabolic subgroup is solvable (see \cite[Proposition 4.1(a)]{LX}).
\item[{\rm (b)}] In Case 2, the group $G_1 \leqs {\rm PGL}_{4}(q)$ is
\[
G_1 = \left\{\begin{array}{ll}
{\rm PSL}_{4}(q).2 & \mbox{if $q \equiv 1 \imod{4}$} \\
{\rm PSL}_{4}(q) & \mbox{otherwise.}
\end{array}\right.
\]
\item[{\rm (c)}] In Case 6, $N_1$ denotes the stabilizer in $G_1$ of a $1$-dimensional nondegenerate subspace of the natural module.
\item[{\rm (d)}] In Case 7, we define $e=2$ if $q^m \equiv 3 \imod{4}$, otherwise $e=1$. In addition, we write $N_{1}^{-}$ for the stabilizer of a nondegenerate $1$-space $W$ such that $W^{\perp}$ is of minus-type (that is, $W^{\perp}$ has Witt index $m-1$).
\item[{\rm (e)}] In Cases 8 and 9, $N_1$ denotes the stabilizer of a $1$-space $U$, where $U$ is nondegenerate if $q$ is odd and nonsingular if $q$ is even.
\end{itemize}
\end{remk}

\begin{table}
\[
\begin{array}{clllc} \hline
\mbox{Case} &  \hspace{4.7mm} G_1 & H_1 & K_1  & \ell \\ \hline
1 & \hspace{4.7mm} {\rm PGL}_{n}(q) & \frac{q^n-1}{q-1} & P_1 & \frac{q^n-1}{q-1} \\

2  & \hspace{4.7mm} \mbox{See Remark \ref{r:1}(b)}
& q^3{:}\frac{q^3-1}{(2,q-1)} & {\rm PGSp}_{4}(q) & \frac{q^3(q^3-1)}{(2,q-1)} \\

3  & \hspace{4.7mm} {\rm Sp}_{2m}(q) & q^m{:}(q^m-1) & {\rm O}_{2m}^{-}(q) & q^m(q^m-1) \\

4  & \hspace{4.7mm} {\rm Sp}_{4}(q) & q^2{:}(q^2-1) & {\rm Sp}_{2}(q^2).2 & q^2(q^2-1) \\

5  & \hspace{4.7mm} {\rm PGSp}_{4}(q) & q^{1+2}{:}(q^2-1) & {\rm PGSp}_{2}(q^2).2 & q^3(q^2-1) \\

6  & \hspace{4.7mm} {\rm PGU}_{2m}(q) & q^{2m}{:}\frac{q^{2m}-1}{q+1} & N_1 & \frac{q^{2m}(q^{2m}-1)}{q+1} \\

7  & \hspace{4.7mm} {\rm SO}_{2m+1}(q) & (q^{\frac{1}{2}m(m-1)}.q^m){:}\frac{q^m-1}{e} & N_{1}^{-} & \frac{1}{e}q^{\frac{1}{2}m(m+1)}(q^m-1) \\

8  & \left\{\begin{array}{ll}
\O_{2m}^{+}(q) & \mbox{$q$ even} \\
{\rm PSO}_{2m}^{+}(q) & \mbox{$q$ odd}
\end{array}\right.
& q^m{:}\frac{q^m-1}{(2,q-1)} & N_1 & \frac{q^{m}(q^m-1)}{(2,q-1)} \\

9  & \left\{\begin{array}{ll}
\O_{8}^{+}(q) & \mbox{$q$ even} \\
{\rm PSO}_{8}^{+}(q) & \mbox{$q$ odd}
\end{array}\right.
& q^4{:}\frac{q^4-1}{(2,q-1)} & N_1 & \frac{q^4(q^4-1)}{(2,q-1)} \\ \hline
\end{array}
\]
\caption{Type I factorizations $G_1=H_1K_1$ with $|H_1|=\ell$ in Theorem \ref{t:main}}
\label{tab2}
\end{table}

\begin{remk}\label{r:2}
It is interesting to compare the structure of the subgroups $H_1$ in Table \ref{tab2} with the overgroup $A$ of $H \cap G_0$ presented in Table \ref{tab}. With the exception of Case 1, the overgroup $A$ contains the full unipotent radical $U$ of an appropriate parabolic subgroup of $G_0$, but in most cases we find that only a small subgroup of $U$ is needed to construct a solvable factor. Indeed, the main challenge in the proof of Theorem \ref{t:main} is to determine how much of the unipotent radical needs to be included in a solvable factor. In this respect, it is worth highlighting Cases 2, 5 and 7, which are special in the sense that every solvable factor must contain the full unipotent radical (see the proofs of Propositions \ref{p:case2} and \ref{p:case7} for further details).
\end{remk}

Next we turn to the problem of determining an absolute lower bound on the order of a solvable factor of any almost simple group with a given socle. Let $G_0$ be a nonabelian finite simple group and let
\[
\mathcal{S}(G_0) = \{ H \,:\, \mbox{$G = HK$, $H$ solvable, $K$ core-free, ${\rm soc}(G) = G_0$}\}
\]
be the set of solvable factors of almost simple groups with socle $G_0$. Set
\[
\ell(G_0) = \left\{\begin{array}{ll}
\min\{|H|\,:\, H \in \mathcal{S}(G_0)\} & \mbox{if $\mathcal{S}(G_0)$ is nonempty} \\
0 & \mbox{otherwise.}
\end{array}\right.
\]
By inspecting \cite{LX}, we see that $\mathcal{S}(G_0)$ is nonempty if and only if one of the following holds:
\begin{itemize}\addtolength{\itemsep}{0.2\baselineskip}
\item[{\rm (a)}] $G_0 \in \{ A_n, {\rm PSL}_{n}(q), \O_{2m+1}(q), {\rm P\O}_{2m}^{+}(q), {\rm M}_{11}, {\rm M}_{12},{\rm M}_{22},{\rm M}_{23},{\rm M}_{24},{\rm J}_{2}, {\rm HS}, {\rm He}, {\rm Suz}\}$;
\item[{\rm (b)}] $G_0 = {\rm PSU}_{n}(q)$ and either $n$ is even, or $n=3$ and $q \in \{3,5,8\}$;
\item[{\rm (c)}] $G_0 = {\rm PSp}_{2m}(q)$ and either $q$ is even, or $m=2$, or $m=q=3$.
\end{itemize}
By combining Theorem \ref{t:main} with the analysis in \cite{LX}, we can compute $\ell(G_0)$ precisely (see Remark \ref{r:abelian} for comments on the analogous problem for abelian factors).

\begin{theorem}\label{t:main2}
Let $G_0$ be a nonabelian finite simple group with $\ell(G_0)>0$. Then $\ell(G_0)$ is
recorded in Table $\ref{tab:ell}$, together with a triple $(G,H,K)$ where $G=HK$ has socle $G_0$, $H$ is solvable, $K$ is core-free and $|H|=\ell(G_0)$.
\end{theorem}

\begin{table}
\[
\begin{array}{lcl} \hline
G_0 & \ell(G_0) & G,\, H, \, K \\ \hline
A_n & n & S_n,\, C_n,\, S_{n-1} \\
{\rm PSL}_{2}(q),\, q \geqs 5 & q+1 & {\rm PGL}_{2}(q),\, C_{q+1},\, P_1 \\
q \ne 5,7,9,11 & & \\
{\rm PSL}_{n}(q),\, n \geqs 3 & \frac{q^n-1}{q-1} & {\rm PGL}_{n}(q),\, \frac{q^n-1}{q-1},\, P_1 \\
(n,q) \ne (4,2) & & \\
{\rm PSU}_{2m}(q),\, m \geqs 2 & \frac{q^{2m}(q^{2m}-1)}{q+1} & {\rm PGU}_{2m}(q),\, q^{2m}{:}\frac{q^{2m}-1}{q+1},\, N_1 \\
(m,q) \ne (2,2), (2,3), (2,8) & & \\
{\rm PSp}_{2m}(q),\, \mbox{$m \geqs 2$, $q$ even} & q^m(q^m-1) & {\rm PSp}_{2m}(q),\, q^m{:}(q^m-1),\, {\rm O}_{2m}^{-}(q) \\
(m,q) \ne (2,2) & & \\
{\rm PSp}_{4}(q),\, \mbox{$q \geqs 5$ odd} & q^3(q^2-1) & {\rm PGSp}_{4}(q),\,  q^{1+2}{:}(q^2-1),\, {\rm PGSp}_{2}(q^2).2 \\
\O_{2m+1}(q),\, m \geqs 3 & \frac{1}{e}q^{\frac{1}{2}m(m+1)}(q^m-1) & \mbox{See Remark \ref{r:3}(a)} \\
{\rm P\O}_{2m}^{+}(q),\, m \geqs 4 & \frac{1}{d}q^m(q^m-1) & \mbox{See Remark \ref{r:3}(b)} \\
& &  \\
{\rm PSL}_{2}(7) & 7 & {\rm PSL}_{2}(7),\, C_{7},\, S_4 \\
{\rm PSL}_{2}(11) & 11 & {\rm PSL}_{2}(11),\, C_{11},\, A_5 \\
{\rm PSU}_{3}(3) & 216 & {\rm PSU}_{3}(3),\, 3^{1+2}_{+}{:}8,\, {\rm PSL}_{2}(7)  \\
{\rm PSU}_{3}(5) & 1000 & {\rm PSU}_{3}(5),\, 5^{1+2}_{+}{:}8,\, A_7 \\
{\rm PSU}_{3}(8) & 513 & {\rm PSU}_{3}(8).3^2,\, 57{:}9,\, 2^{3+6}{:}(63{:}3) \\
{\rm PSU}_{4}(3) & 162 & {\rm PSU}_{4}(3).2,\, 3^4{:}2,\, {\rm PSL}_{3}(4).2 \\
{\rm PSU}_{4}(8) & 4617 & {\rm PSU}_{4}(8).3,\, (513{:}3).3,\, (2^{12}.{\rm SL}_{2}(64).7).3 \\
{\rm PSp}_{4}(3) & 27 & {\rm PSp}_{4}(3),\, 3_{+}^{1+2},\, 2^4{:}A_5 \\
{\rm M}_{11} & 11 & {\rm M}_{11},\, C_{11},\, {\rm M}_{10} \\
{\rm M}_{12} & 12 & {\rm M}_{12},\, D_{12},\, {\rm M}_{11} \\
{\rm M}_{22} & 22 & {\rm M}_{22}.2, \, D_{22},\, {\rm PSL}_{3}(4).2 \\
{\rm M}_{23} & 23 & {\rm M}_{23},\, C_{23},\, {\rm M}_{22} \\
{\rm M}_{24} & 24 & {\rm M}_{24},\, S_4,\, {\rm M}_{23} \\
{\rm J}_{2} & 100 & {\rm J}_{2}.2,\, 5^2{:}4,\, G_2(2) \\
{\rm HS} & 100 & {\rm HS}.2,\, 5^2{:}4,\, {\rm M}_{22}.2 \\
{\rm He} & 2058 & {\rm He}.2,\, 7_{+}^{1+2}{:}6,\, {\rm Sp}_{4}(4).4 \\
{\rm Suz} & 2916 & {\rm Suz}.2,\, 3^5{:}12,\, G_2(4).2 \\ \hline
\end{array}
\]
\caption{The values of $\ell(G_0)$ in Theorem \ref{t:main2} and a factorization $G=HK$ with ${\rm soc}(G) = G_0$ and $|H|=\ell(G_0)$}
\label{tab:ell}
\end{table}

\begin{remk}\label{r:3}
Some comments on the information in Table \ref{tab:ell} are in order.
\begin{itemize}\addtolength{\itemsep}{0.2\baselineskip}
\item[{\rm (a)}] For $G_0 = \O_{2m+1}(q)$, we define the integer $e \in \{1,2\}$ as in Remark \ref{r:1}(d) and we take
\[
G = {\rm SO}_{2m+1}(q),\;\; H = (q^{\frac{1}{2}m(m-1)}.q^m){:}\frac{q^m-1}{e},\;\; K= N_1^{-}.
\]
\item[{\rm (b)}] For $G_0 = {\rm P\O}_{2m}^{+}(q)$ we set $d=(2,q-1)$ and
\[
G = \left\{\begin{array}{ll}
\O_{2m}^{+}(q) & \mbox{$q$ even} \\
{\rm PSO}_{2m}^{+}(q) & \mbox{$q$ odd,}
\end{array}\right. \;\; H= q^m{:}\frac{q^m-1}{d}, \;\; K=N_1.
\]
\item[{\rm (c)}] In the second row of the table, we have $G_0 = {\rm PSL}_{2}(q)$ with $q \geqs 5$ and $q \ne 5,7,9,11$. Note that $q \ne 5,7,9$ due to the existence of the exceptional isomorphisms ${\rm PSL}_{2}(5) \cong A_5$, ${\rm PSL}_{2}(7) \cong {\rm PSL}_{3}(2)$ and ${\rm PSL}_{2}(9) \cong A_6$. Similarly, we note the isomorphisms ${\rm PSL}_{4}(2) \cong A_8$ and ${\rm PSU}_{4}(2) \cong {\rm PSp}_{4}(3)$ (see \cite[Proposition 2.9.1]{KL}).
\item[{\rm (d)}] For $G_0 = {\rm PSU}_{4}(3)$ we take $G = {\rm PSU}_{4}(3).2 < {\rm PGU}_{4}(3)$.
\end{itemize}
\end{remk}

In \cite{LPS-4}, Liebeck, Praeger and Saxl determine the almost simple primitive permutation groups with a regular subgroup. In other words, they determine the exact nontrivial factorizations of almost simple groups with a maximal factor, and one can read off the examples with a solvable factor. Our next result describes the exact factorizations of almost simple groups with a solvable factor in a more general setting, without assuming that one of the factors is maximal.

In Table~$\ref{tab-sharp}$ we write $\mathcal{O} \leqs C_2$, and $\mathcal{O}_1$, $\mathcal{O}_2$ are subgroups of $\mathcal{O}$ such that $\mathcal{O} = \mathcal{O}_1\mathcal{O}_2$. Also recall that a subgroup of $S_n$ is \emph{$2$-homogeneous} if it acts transitively on the set of subsets of $\{1, \ldots, n\}$ of size $2$.

\begin{theorem}\label{t:main3}
Let $G$ be an almost simple group with socle $G_0$ and $G=HK$ such that $H$ is solvable, $K$ is core-free and $H\cap K=1$. Then one of the following holds:
\begin{itemize}\addtolength{\itemsep}{0.2\baselineskip}
\item[{\rm (i)}] $G_0 = A_n$, $H$ is transitive on $\{1, \ldots, n\}$ and $K = A_{n-1}$ or $S_{n-1}$.
\item[{\rm (ii)}] $G_0 = A_n$, $n=p^a$ is a prime power, $H \leqs {\rm A\Gamma L}_{1}(p^a)$ is $2$-homogeneous on $\{1, \ldots, n\}$ and $A_{n-2} \normeq K \leqs S_{n-2} \times S_2$.
\item[{\rm (iii)}] $G_0 = {\rm PSL}_{n}(q)$, $H \cap G_0 \leqs \frac{q^n-1}{(q-1)d}{:}n$ and $q^{n-1}{:}{\rm SL}_{n-1}(q) \normeq K \cap G_0$ with $d=(n,q-1)$.
\item[{\rm (iv)}] $G_0 = {\rm PSp}_{2m}(q)$ with $q$ even, $m \geqs 3$, $H \cap G_0 \leqs q^{m}{:}(q^m-1).m$ and $K \cap G_0 = \O_{2m}^{-}(q)$.
\item[{\rm (v)}] $(G,H,K)$ is one of the cases listed in Table~$\ref{tab-sharp}$.
\end{itemize}
\end{theorem}

\begin{table}
\[
\begin{array}{clll} \hline
\mbox{Case} & G & H & K \\ \hline
1 & A_8.\mathcal{O} & {\rm AGL}_{1}(8) & (A_5 \times 3).2.\mathcal{O} \\
2 & A_8.\mathcal{O} & {\rm A\Gamma L}_{1}(8) & S_5 \times \mathcal{O} \\
3 & A_{25} & 5^2{:}{\rm SL}_{2}(3) & A_{23} \\
4 & S_{25} & 5^2{:}{\rm SL}_{2}(3) & S_{23}, A_{23} \times 2 \\
5 & A_{32}.\mathcal{O} & {\rm A\Gamma L}_{1}(32) & (A_{29} \times 3).2.\mathcal{O} \\
6 & A_{49} & 7^2{:}Q_8.S_3 & A_{47} \\
7 & S_{49} & 7^2{:}Q_8.S_3 & S_{47}, A_{47} \times 2 \\
8 & A_{121}.\mathcal{O} & 11^2{:}{\rm SL}_{2}(3).5.\mathcal{O} & A_{119} \\
9 & S_{121} & 11^2{:}{\rm SL}_{2}(3).5 & S_{119}, A_{119} \times 2 \\
10 & A_{529} & 23^2{:}{\rm SL}_{2}(3).22 & A_{527} \\
11 & S_{529} & 23^2{:}{\rm SL}_{2}(3).22 & S_{527}, A_{527} \times 2 \\
12 & {\rm PSL}_{2}(11).\mathcal{O} & 11{:}\mathcal{O} & A_5 \\
13 & {\rm PSL}_{2}(11).\mathcal{O} & 11{:}(5 \times \mathcal{O}_1) & A_4.\mathcal{O}_2 \\
14 & {\rm PSL}_{2}(23).\mathcal{O} & 23{:}(11 \times \mathcal{O}) & S_4 \\
15 & {\rm PSL}_{2}(29) & 29{:}7 & A_5 \\
16 & {\rm PSL}_{2}(59).\mathcal{O} & 59{:}(29 \times \mathcal{O}) & A_5 \\
17 & {\rm PSL}_{3}(3).\mathcal{O} & 13{:}(3 \times \mathcal{O}) & {\rm A\Gamma L}_{1}(9) \\
18 & {\rm P\Gamma L}_{3}(4).\mathcal{O} & 7{:}(3 \times \mathcal{O}).S_3 & 2^4{:}(3 \times D_{10}).2 \\
19 & {\rm PSL}_{3}(8).(3 \times \mathcal{O}) & 73{:}(9 \times \mathcal{O}_1) & 2^{3+6}{:}7^2{:}(3 \times \mathcal{O}_2) \\
20 & {\rm PGL}_{4}(3) & (3^3{:}13{:}3).2 & ((4 \times {\rm PSL}_{2}(9)){:}2).2 \\
21 & {\rm PSL}_{4}(3).2^2 & (3^3{:}13{:}3).2 & ((4 \times {\rm PSL}_{2}(9)){:}2).2^2 \\
22 & {\rm P\Gamma L}_{4}(4).\mathcal{O} & (2^6{:}63{:}3).2 & ((5 \times {\rm PSL}_{2}(16)){:}2).2.\mathcal{O} \\
23 & {\rm PSL}_{5}(2).\mathcal{O} & 31{:}(5 \times \mathcal{O}) & 2^6{:}(S_3 \times {\rm PSL}_{3}(2)) \\
24 & {\rm PSU}_{3}(8).3^2.\mathcal{O} & 57{:}9.\mathcal{O}_1 & 2^{3+6}{:}(63{:}3).\mathcal{O}_2 \\
25 & {\rm PSU}_{4}(3).2 & 3^4{:}2,3^3.6, 3^3.D_6, 3_{\pm}^{1+2}.6, 3^{1+2}_{\pm}.D_6,  & {\rm PSL}_{3}(4).2 \\
26 & {\rm PGU}_{4}(3) & 3^4{:}4 & {\rm PSL}_{3}(4).2 \\
27 & {\rm PSU}_{4}(3).2^2 & 3^4{:}2^2,3^3.D_{12}, 3^3.(6 \times 2), 3^2.3^2.2^2 & {\rm PSL}_{3}(4).2 \\
28 & {\rm PSU}_{4}(3).2^2 & 3^4{:}2,3^3.6, 3^3.D_6, 3_{\pm}^{1+2}.6, 3_{\pm}^{1+2}.D_6 & {\rm PSL}_{3}(4).2^2 \\
29 & {\rm PSU}_{4}(3).D_8 &  3^4{:}2^2,  3^4{:}4, 3^3.D_{12}, 3^3.(6 \times 2), 3^2.3^2.2^2 & {\rm PSL}_{3}(4).2^2 \\
30 & {\rm PSU}_{4}(8).3.\mathcal{O} & (513{:}3).3.\mathcal{O}_1 & (2^{12}.{\rm SL}_{2}(64).7).3.\mathcal{O}_2 \\
31 & {\rm PSp}_{4}(3).\mathcal{O} & 3_{\pm}^{1+2}{:}\mathcal{O}_1 & 2^4{:}(A_5.\mathcal{O}_2) \\
32 & {\rm PSp}_{4}(3).\mathcal{O} & 3_{+}^{1+2}{:}(Q_8.\mathcal{O}_1) & \mathcal{O}_2.S_5 \\
33 & {\rm PGSp}_{4}(3) & 2^4{:}5{:}4 & 3_{+}^{1+2}{:}S_3, 3^3{:}S_3 \\
34 & {\rm PGSp}_{6}(3) & (3^{1+4}_{+}{:}2^{1+4}.D_{10}).2 & {\rm PSL}_{2}(27).6 \\
35 & \O_{8}^{+}(2).\mathcal{O} & 2^6{:}15, 2^4{:}15.4 & A_9.\mathcal{O} \\ 
36 & {\rm M}_{11} & C_{11} & {\rm M}_{10} \\
37 & {\rm M}_{11} & 11{:}5 & (3^2{:}Q_8).2 \\
38 & {\rm M}_{12}.\mathcal{O} & (3^2{:}Q_8).2 & {\rm PSL}_{2}(11).\mathcal{O} \\
39 & {\rm M}_{12} & C_6 \times C_2, A_4,  D_{12} & {\rm M}_{11} \\
40 & {\rm M}_{12}.2 & S_4,  D_{24},  D_8 \times 3, 3{:}D_8 & {\rm M}_{11} \\
41 & {\rm M}_{22}.2 & D_{22} & {\rm PSL}_{3}(4).2 \\
42 & {\rm M}_{23} & C_{23} & {\rm M}_{22} \\
43 & {\rm M}_{23} & 23{:}11 & {\rm PSL}_{3}(4).2, 2^4{:}A_7 \\
44 & {\rm M}_{24} & S_4, D_{24}, D_8 \times 3, 3{:}D_8, A_4 \times 2 & {\rm M}_{23} \\
45 & {\rm J}_{2}.2 & 5^2{:}4 & G_2(2) \\
46 & {\rm HS}.2 & 5^2{:}4 & {\rm M}_{22}.2 \\
47 & {\rm He}.2 & 7_{+}^{1+2}{:}6 & {\rm Sp}_{4}(4).4 \\ \hline
\end{array}
\]
\caption{The exact factorizations $G=HK$ in Theorem \ref{t:main3}(v)}
\label{tab-sharp}
\end{table}

\begin{remk}\label{r:4}
Let us comment on the statement of Theorem \ref{t:main3}.
\begin{itemize}\addtolength{\itemsep}{0.2\baselineskip}
\item[{\rm (a)}] As recorded in Table \ref{t:exact}, there exist infinite families of exact factorizations satisfying the conditions in cases (i)--(iv). Note that in the final row of Table \ref{t:exact}, we have $H = W{:}C<P_m$, where $W$ is elementary abelian of order $q^m$ and $C$ is a Singer cycle in ${\rm GL}_{m}(q)$ that acts transitively on the nontrivial elements of $W$. In case (iv), we are not aware of an example with $m$ even.

\begin{table}
\[
\begin{array}{ccccc} \hline
{\rm Case} & G & H & K & \mbox{Conditions} \\ \hline
{\rm (i)} & S_n & C_n & S_{n-1} & \\
{\rm (ii)} & S_{p^a} & {\rm AGL}_{1}(p^a) & S_{p^a-2} & \\
{\rm (iii)} & {\rm PGL}_{n}(q) & \frac{q^n-1}{q-1} & P_1 & \\
{\rm (iv)} & {\rm Sp}_{2m}(q) & q^m{:}(q^m-1) & \O_{2m}^{-}(q) & \mbox{$m \geqs 3$ odd} \\ \hline
\end{array}
\]
\caption{Some exact factorizations in parts (i)--(iv) of Theorem \ref{t:main3}}
\label{t:exact}
\end{table}

\item[{\rm (b)}] In Table~\ref{tab-sharp}, we have omitted any factorizations that are  isomorphic to one of the examples in parts (i)--(iv). For instance, $G=S_6$ has an exact factorization $G = HK$ with $H \in \{C_6,D_6\}$ and $K = {\rm PGL}_{2}(5) \cong S_5$, which is included in part (i). Similarly, $G=A_8 \cong {\rm PSL}_{4}(2)$ has a factorization with $H = C_{15}$ and $K = {\rm AGL}_{3}(2)$, which is covered in part (iii).

\item[{\rm (c)}] Notice that there are several cases in Table~\ref{tab-sharp} where both $H$ and $K$ are solvable. Here we record the triple $(G,H,K)$ in the table, but we omit $(G,K,H)$.

\item[{\rm (d)}] Consider Case 19 in Table \ref{tab-sharp}, with $G = {\rm PSL}_{3}(8).(3 \times \mathcal{O})$, $H = 73{:}(9 \times \mathcal{O})$ and $K = 2^{3+6}{:}7^2{:}3$. Then $G = HK$ is an exact factorization and $H$ is a maximal subgroup of $G$, so we can view $G$ as an almost simple primitive group with point stabilizer $H$ and regular subgroup $K$. It is worth noting that this case has been omitted in the statement of \cite[Theorem 1.1]{LPS-4}, which describes all the almost simple primitive groups with a regular subgroup.

\item[{\rm (e)}] In Case 25 of Table \ref{tab-sharp}, $G = {\rm PSU}_{4}(3).2$ is the unique index-two subgroup of ${\rm PGU}_{4}(3)$. In Case 27, the group $G = {\rm PSU}_{4}(3).2^2$ contains an involutory graph automorphism with centralizer of type ${\rm Sp}_{4}(3)$, which distinguishes it from the group in Case 28 of the same shape.
\end{itemize}
\end{remk}

In a sequel, we will use the main results in this paper to characterize the quasiprimitive permutation groups which contain a transitive solvable subgroup, extending classical results of Burnside and Schur on primitive groups with a transitive cyclic subgroup. As a first step in this direction, we present Theorem \ref{nil-B-gps} below, which describes the finite primitive permutation groups with a regular nilpotent subgroup. This extends earlier work of Li \cite{Li-Abel-gps} on primitive groups containing an abelian regular subgroup. In order to state this result, we need to introduce some new terminology.

Let $H$ be a finite group. By considering the regular action, we may view $H$ as a regular subgroup of a $2$-transitive permutation group (namely, $S_n$ acting on $\{1, \ldots, n\}$, where $n$ is the order of $H$). Following Wielandt, we say that $H$ is a \emph{Burnside-group} (in short, \emph{B-group}) if \emph{every} primitive permutation group containing a regular subgroup isomorphic to $H$ is $2$-transitive (see \cite[Definition\,25.1]{Wielandt-book}). The term `B-group' is due to the fact that Burnside \cite{burnside} first proved that every cyclic group of composite prime-power order is a B-group (also see \cite{Wildon}).
For more than a century, the problem of classifying B-groups and characterizing primitive groups with  a regular subgroup has attracted considerable attention; we refer the reader to \cite{Wielandt-book} for the early work and \cite{Bau1,Li-Abel-gps,LPS-3,LPS-4} for more recent results.

For some special classes of groups, the above problems have been solved. For example, Schur \cite{schur} proved that every cyclic group of composite order is a B-group, and the
primitive groups containing a cyclic regular subgroup are classified in \cite{Jones}. More generally, the primitive groups with an abelian regular subgroup are determined in \cite{Li-Abel-gps}. Note that a group of the form
\[
H=H_1\times H_2\times\dots\times H_d
\]
with $|H_1|=\cdots=|H_d|=m\geqslant3$ and $d\geqslant 2$ is not a B-group, since
$H$ is a regular subgroup of the simply primitive permutation group $S_m\wr S_d$ of degree $m^d$.

Let $G \leqs {\rm Sym}(\Delta)$ be a primitive group with socle $T$ and let $d$ be a positive integer. Consider the primitive product action of $G \wr S_d$ on $\Delta^d$. Then we refer to any primitive subgroup of $G \wr S_d$ with socle $T^d$ as a \emph{blow-up} of $G$ on $\Delta$ (note that we allow $d=1$). This leads us naturally to the following definition.

\begin{def-non}\label{generalized B-group}
A finite group $H$ is called a \emph{generalized B-group} if every primitive permutation
group containing a regular subgroup isomorphic to $H$ is a blow-up of a $2$-transitive group.
\end{def-non}

An elementary abelian $p$-group $C_p^d$ is a regular subgroup of a primitive affine group, which may not be $2$-transitive, and thus $C_p^d$ is not a B-group, nor a generalized B-group, in general. More precisely, one can show that $C_p^d$ is a B-group if and only if $p=2$, $2^d-1$ is a Mersenne prime and the only simple irreducible subgroups of ${\rm GL}_{d}(2)$ are $C_{2^d-1}$ and ${\rm GL}_{d}(2)$ itself (see \cite[Proposition 7.5]{Wildon}). For example, $C_2^3$ is a B-group.

We will also need the following definition.

\begin{def-non}\label{class-mate}
Two groups are \emph{class-mates} if they are isomorphic to regular subgroups of the same primitive permutation group, which is not a symmetric or alternating group in its natural action.
\end{def-non}

Note that if we allow the symmetric or alternating group in its natural action, then two groups would be class-mates if and only if they have the same order, so we exclude this situation in the definition.

By \cite{Li-Abel-gps}, every finite abelian group is either a generalized B-group or a class-mate of an elementary abelian $p$-group. The latter examples come from affine permutation groups; such a group contains a regular subgroup which is elementary abelian and normal, and it may also contain other regular subgroups. Indeed, a very nice construction of Hegedus \cite{Hegedus} shows that if $d \geqs 5$ then ${\rm AGL}_d(p)$ has a regular subgroup which contains no nontrivial translations. In particular, two class-mates may intersect trivially.

Our main result on nilpotent regular subgroups of primitive groups is the following, which extends the main theorem of \cite{Li-Abel-gps} on abelian regular subgroups. In the statement, we say that a primitive group is of \emph{product-type} if it is the blow-up of an almost simple primitive group.

\begin{theorem}\label{nil-B-gps}
Let $G$ be a finite primitive permutation group of degree $n$ with socle $T$ and let $H$ be a nilpotent regular subgroup of $G$. Then one of the following holds:
\begin{itemize}\addtolength{\itemsep}{0.2\baselineskip}
\item[{\rm (i)}] $T = A_n$ and $H$ is a nilpotent group of order $n$.
\item[{\rm (ii)}] $G \leqs {\rm AGL}_{d}(p)$ is an affine group and $H$ is a class-mate of $T = C_p^d$.
\item[{\rm (iii)}] $G$ is a product-type group, $T=S^d$, $n=m^d$ and one of the following holds:  

\vspace{1mm}

\begin{itemize}\addtolength{\itemsep}{0.2\baselineskip}
\item[{\rm(a)}] $H$ is a class-mate of $C_m^d$ and $(S,m)$ is one of 
\[
\mbox{$(A_m,m)$, 
$({\rm PSL}_{a}(q), \frac{q^a-1}{q-1})$, $({\rm PSL}_{2}(11),11)$, $({\rm M}_{11},11)$ or $({\rm M}_{23},23)$.}
\]
\item[{\rm(b)}] $S = {\rm M}_{12}$ and $H$ is a class-mate of $(C_6 \times C_2)^d$.
\item[{\rm(c)}] $S = {\rm M}_{24}$ and $H$ is a class-mate of $(D_8\times C_3)^d$.
\end{itemize}
\item[{\rm (iv)}] $G$ is a product-type group with $T = {\rm PSp}_4(3)^d$ and $H$ is a class-mate of $(3_+^{1+2})^d$. 
\end{itemize}
\end{theorem}

\begin{remk}\label{r:5}
Consider the primitive group $G = {\rm PSp}_4(3)\wr S_d$ in its product action of degree $27^d$. Here $(3_+^{1+2})^d$ is a regular subgroup of $G$, and so is $(3_{-}^{1+2})^d$. As explained in Example~\ref{3^2:3}, if $e,f$ are non-negative integers and $d=e+3f$ then $(3_+^{1+2})^e\times(3^6{:}3_+^{1+2})^f$ is another regular subgroup. Since the action of ${\rm PSp}_{4}(3)$ on the cosets of $2^4{:}A_5$ is simply primitive, it follows that $H$ is \emph{not} a generalized B-group in part (iv).
\end{remk}

This leads to the following theorem, which shows that `most' nilpotent groups are generalized B-groups and class-mates of homocyclic groups.

\begin{theorem}\label{t:5}
Let $H$ be a finite nilpotent group. Then one of the following holds:
\begin{itemize}\addtolength{\itemsep}{0.2\baselineskip}
\item[{\rm(i)}] $H$ is a generalized B-group.
\item[{\rm(ii)}] $H$ is a class-mate of an elementary abelian $p$-group.
\item[{\rm(iii)}] $H$ is a class-mate of $(3_{+}^{1+2})^d$ for some $d \geqs 1$.
\end{itemize}
\end{theorem}

Our notation is all fairly standard. Let $A$ and $B$ be finite groups and let $n$ and $p$ be positive integers, with $p$ a prime. We denote a cyclic group of order $n$ by $C_n$ (or just $n$) and $A^n$ is the direct product of $n$ copies of $A$. In particular, $C_p^n=p^n$ is an elementary abelian $p$-group of order $p^n$. If $p$ is an odd prime, then $p^{1+2m}_{+}$ denotes an extraspecial group of order $p^{1+2m}$ and exponent $p$, while $p^{1+2m}_{-}$ has exponent $p^2$. A split extension of $A$ by $B$ will be denoted by writing $A{:}B$, whereas $A.B$ is an unspecified extension. We write ${\rm soc}(A)$ for the socle of $A$ (that is, the subgroup of $A$ generated by its minimal normal subgroups). In addition, we adopt the notation from \cite{KL} for finite simple groups.

\vs

Finally, let us say a few words on the organization of the paper. The next section, which comprises the main bulk of the paper, is dedicated to the proof of Theorem \ref{t:main}, with separate subsections to handle the cases where $G$ is a linear, symplectic, unitary and orthogonal group. A proof of Theorem \ref{t:main2} is presented in Section \ref{s:main2}, which includes a detailed analysis of the Type II and III factorizations of almost simple classical groups. Finally, the proof of Theorem \ref{t:main3} is given in 
Section \ref{s:main3}, and Theorems \ref{nil-B-gps} and \ref{t:5} are established in Section \ref{s:main4}.

\vs

\noindent \textbf{Acknowledgements.} This work was partially supported by NSFC grants no. 11231008 and no. 11771200. Burness thanks the Department of Mathematics at the Southern University of Science and Technology (SUSTech) for their generous hospitality during a research visit in April 2019. He also thanks Bob Guralnick for helpful discussions.

\section{Proof of Theorem \ref{t:main}}\label{s:main1}

In this section we present a proof of Theorem \ref{t:main}. We will partition the proof into four subsections, according to the socle of the almost simple classical groups we are considering.

Let $q$ be a prime power and let $m \geqs 2$ be an integer. Recall that a prime $r$ is a \emph{primitive prime divisor} of $q^m-1$ if $r$ divides $q^m-1$, but it does not divide $q^i-1$ for $1 \leqs i < m$. By a theorem of Zsigmondy \cite{Zsig}, such a prime always exists, unless $(m,q) = (6,2)$, or $m=2$ and $q$ is a Mersenne prime.

\subsection{Linear groups}\label{ss:lin}

We begin the proof of Theorem \ref{t:main} by handling Case 1 in Table \ref{tab}. Recall that a \emph{Singer cycle} in ${\rm PGL}_{n}(q)$ is a cyclic subgroup of order $\frac{q^n-1}{q-1}$ (equivalently, it is an irreducible cyclic subgroup of maximal order, with respect to the natural module).

\begin{prop}\label{p:case1}
Let $G$ be an almost simple group with socle $G_0 = {\rm PSL}_{n}(q)$ and subgroups $H$ and $K$ satisfying the conditions in Case $1$ of Table $\ref{tab}$.
\begin{itemize}\addtolength{\itemsep}{0.2\baselineskip}
\item[{\rm (i)}] If $G=HK$, then $|H|$ is divisible by $\frac{q^n-1}{q-1}$.
\item[{\rm (ii)}] If $G = {\rm PGL}_{n}(q)$ and $K=P_1$, then $G=HK$ with $H = \frac{q^n-1}{q-1}$ a Singer cycle.
\end{itemize}
\end{prop}

\begin{proof}
Here $K \cap G_0 \leqs P<G_0$, where $P$ is the stabilizer in $G_0$ of a $1$-dimensional (or $(n-1)$-dimensional) subspace of the natural module. In particular, $G=HK$ only if $|G_0:P|= \frac{q^n-1}{q-1}$ divides $|H|$. This establishes part (i). The existence of the factorization in part (ii) follows immediately from the fact that a Singer cycle in $G = {\rm PGL}_{n}(q)$ acts transitively on the set of $1$-dimensional subspaces of the natural module for $G$.
\end{proof}

\begin{prop}\label{p:case2}
Let $G$ be an almost simple group with socle $G_0 = {\rm PSL}_{4}(q)$ and subgroups $H$ and $K$ satisfying the conditions in Case $2$ of Table $\ref{tab}$.
\begin{itemize}\addtolength{\itemsep}{0.2\baselineskip}
\item[{\rm (i)}] If $G=HK$, then $|H|$ is divisible by $\frac{q^3(q^3-1)}{(2,q-1)}$.
\item[{\rm (ii)}] Let $G \leqs {\rm PGL}_{4}(q)$ and $K = {\rm PGSp}_{4}(q)<G$, where
\[
G = \left\{\begin{array}{ll}
{\rm PSL}_{4}(q).2 & \mbox{if $q \equiv 1 \imod{4}$} \\
{\rm PSL}_{4}(q) & \mbox{otherwise.}
\end{array}\right.
\]
Then $G=HK$ with $H = q^3{:}\frac{q^3-1}{(2,q-1)} < P_k$ and $k \in \{1,3\}$.
\end{itemize}
\end{prop}

\begin{proof}
Let $P = U{:}L$ be the stabilizer in $G_0$ of a $1$-dimensional (or $3$-dimensional) subspace of the natural module, where $U$ is the unipotent radical and $L = \frac{1}{d}{\rm GL}_{3}(q)$ is a Levi factor, with $d=(4,q-1)$. We may identify $U$ with the natural module for $L$. According to Table \ref{tab} we have
\[
H \cap G_0 \leqs q^3{:}\left(\frac{q^3-1}{d}{:}3\right)<P
\]
and we may as well assume $K$ is maximal among subgroups of $G$ with $K \cap G_0 \trianglerighteqslant {\rm PSp}_{4}(q)$. Set $e=(2,q-1)$ and note that $\frac{1}{e}q^2(q^3-1)$ divides $|G:K|$.

Suppose $G=HK$, in which case $|H|$ is divisible by $\frac{1}{e}q^2(q^3-1)$ and $H$ contains an element $x$ of order $r$, a primitive prime divisor of $q^3-1$. Then $\la x \ra$ acts irreducibly on $U$ and it does not normalize a proper nontrivial subgroup of $U$. Since $H \cap U \ne 1$, this forces $U \leqs H$ and thus $|H|$ is divisible by $\frac{1}{e}q^3(q^3-1)$ as in part (i).

Now let us turn to part (ii). First assume $G = {\rm PSL}_{4}(q)$ with $q \equiv 3 \imod{4}$ and take $H = q^3{:}\frac{q^3-1}{2}$ and $K = {\rm PGSp}_{4}(q)$ (see \cite[Proposition 4.8.3]{KL}). By arguing as in the proof of \cite[Proposition 5.9]{LX}, noting that the cyclic subgroup of order $\frac{1}{2}(q^3-1)$ acts semiregularly on the set of nonzero vectors in the natural module for $L$, we deduce that $|H \cap K|=q$ and thus $G=HK$ (here we can appeal to \cite[Proposition 5.9]{LX} since $G \cong {\rm P\O}_{6}^{+}(q)$). For $q \equiv 1 \imod{4}$, an entirely similar argument shows that $G = {\rm PSL}_{4}(q).2< {\rm PGL}_{4}(q)$ admits a factorization $G=HK$ with $H = q^3{:}\frac{q^3-1}{2}$ and $K = {\rm PGSp}_{4}(q)$. Finally, the case $q$ even follows directly from \cite[Proposition 5.9]{LX}.
\end{proof}

\subsection{Symplectic groups}\label{ss:symp}

In this section we handle Cases 3, 4 and 5 in Table \ref{tab}.

\begin{prop}\label{p:case3}
Let $G$ be an almost simple group with socle $G_0 = {\rm Sp}_{2m}(q)$, where $m \geqs 2$ and $q$ is even. Let $H,K$ be subgroups satisfying the conditions in Case $3$ of Table $\ref{tab}$.
\begin{itemize}\addtolength{\itemsep}{0.2\baselineskip}
\item[{\rm (i)}] If $G=HK$, then $|H|$ is divisible by $q^m(q^m-1)$.
\item[{\rm (ii)}] If $G = {\rm Sp}_{2m}(q)$ and $K = {\rm O}_{2m}^{-}(q)$, then $G=HK$ with $H = q^m{:}(q^m-1)<P_m$.
\end{itemize}
\end{prop}

\begin{proof}
Let $V$ be the natural module for $G_0$ and fix a standard symplectic basis
\[
\{e_1, \ldots, e_m, f_1, \ldots, f_m\}
\]
with respect to the underlying symplectic form $(\, , \,)$ on $V$. Let $P$ be the stabilizer in $G_0$ of the maximal totally isotropic subspace
$\la e_1, \ldots, e_m\ra$ and write $P = U{:}L$, where the unipotent radical $U = q^{\frac{1}{2}m(m+1)}$ is elementary abelian and $L = {\rm GL}_{m}(q)$ is a Levi factor. Then
\[
H \cap G_0 \leqs q^{\frac{1}{2}m(m+1)}{:}\left((q^m-1){:}m\right) < P
\]
and we may assume $K$ is maximal among subgroups of $G$ with $K \cap G_0 = {\rm O}_{2m}^{-}(q)$. Note that $|G:K|$ is divisible by $\frac{1}{2}q^m(q^m-1)$. The case $G = {\rm Sp}_{12}(2)$ can be handled using {\sc Magma} \cite{magma}, so we may assume $(m,q) \ne (6,2)$.

Suppose $G=HK$, so $|H|$ is divisible by $\frac{1}{2}q^m(q^m-1)$. Then $H$ contains an element of order $r$, a primitive prime divisor of $q^m-1$, and we note that the smallest nontrivial subgroup of $U$ normalized by $\la x \ra$ has order $q^m$. Therefore, $q^m$ divides $|H|$ and part (i) follows.

Now consider (ii), so $G = {\rm Sp}_{2m}(q)$, $K = {\rm O}_{2m}^{-}(q)$ and $H \leqs U{:}\left((q^m-1){:}m\right)<P$. Let $X$ be the natural module for the Levi factor $L= {\rm GL}_{m}(q)$. Now $U$ is indecomposable and reducible as a module for $L$, with composition series $0 \subseteq W \subseteq U$ where $W = q^{\frac{1}{2}m(m-1)}$. We may identify $W$ and $U/W$ with the modules $\L^2(X)$ and $X$, respectively.

Let $C = \la c \ra <L$ be a Singer cycle and write
$U = W \oplus W'$, where $W'$ is an irreducible $m$-dimensional module for $C$, which we can, and will, identify with $X$.

\vs

\noindent \textbf{Claim.} \emph{We have $G = HK$ with $H = X{:}C = q^m{:}(q^m-1)$.}

\vs

By the proof of \cite[Proposition 5.5]{LX} we have $H \cap K = X\cap K$, so $|X\cap K| \geqs 2$. Since $C$ acts transitively on the set of nontrivial elements in $X$, we may write
\[
X = \{1\} \cup z^C
\]
for some $z \in K$. With respect to an appropriate basis for $V$, we have
\begin{equation}\label{e:z}
z = \left(\begin{array}{cc}
I_m & A \\
0 & I_m
\end{array}\right)
\end{equation}
for a matrix $A$ of size $m$. Moreover, if $s={\rm rank}(A)$ then $z$ has Jordan form $[J_2^s,J_1^{2m-2s}]$ on $V$, where $J_i$ denotes a standard unipotent Jordan block of size $i$.

The conjugacy classes of involutions in $G$ are described by Aschbacher and Seitz in \cite{AS}. In their notation, which is now standard, every involution in $G$ is of type $a$, $b$ or $c$, and the $a$-type involutions $x$ are characterised by the property $(v,v^x) = 0$ for all $v \in V$.
From the description of the elements in $(U \cap K) \setminus W$ given in the proof of \cite[Proposition 5.5]{LX}, it is routine to check that there exists a vector $v \in V$ such that $(v,v^z) \ne 0$, whence $z$ is $G$-conjugate to $b_s$ (if $s$ is odd) or $c_s$ (if $s$ is even).

Set $N = N_L(C) = C{:}D$, where $D = \la \varphi \ra$ has order $m$, and let $\bar{X} = X \otimes \mathbb{F}_{q^m}$. Fix a basis $\{x_1, \ldots, x_m\}$ for $\bar{X}$ such that  $x_i^c = \mu^{q^{i-1}}x_i$ for $i=1, \ldots, m$, where $\mathbb{F}_{q^m}^{\times} = \la \mu \ra$. We may assume that $\varphi$ cyclically permutes the $x_i$.

Observe that $D$ has a $1$-dimensional fixed space on both $X$ and $\bar{X}$. Since $\la u \ra$ is the fixed space of $D$ on $\bar{X}$, where $u=x_1 + \cdots + x_m$, it follows that some nonzero $\mathbb{F}_{q^m}$-scalar multiple of $u$ is contained in $X$. Without loss of generality, we may assume that $u$ itself is in $X$ (the calculation below does not depend on the choice of scalar multiple). This means that $X$ is generated by $u$ (as a module for $C$) and we deduce that
\[
X = \{\a x_1 + \a^qx_2 + \cdots + \a^{q^{m-1}}x_{m} \,:\, \a \in \mathbb{F}_{q^m}\}.
\]
In terms of the basis $\{x_1, \ldots, x_m\}$ for $\bar{X}$, each nontrivial element of $X$ is represented by a nonsingular diagonal matrix. This implies that the matrix $A$ in \eqref{e:z} has rank $m$ and we conclude that all of the involutions in $X$ are of type $b_m$ (if $m$ is odd) or $c_m$ (if $m$ is even).

To complete the proof of the claim, let $n$ denote the number of orbits of $H$ on $\Delta = G/K$. By the orbit counting lemma, we have
\[
n = \frac{1}{|H|}\sum_{h \in H}{\rm fix}(h),
\]
where ${\rm fix}(h)$ is the number of fixed points of $h$ on $\Delta$. Since $H \cap K = X \cap K$, it follows that ${\rm fix}(h) = 0$ for all $h \in H \setminus X$ and thus
\[
n = \frac{1}{|H|}\sum_{x \in X}{\rm fix}(x) = \frac{1}{|H|}\left(|\Delta| + |z^C|\cdot {\rm fix}(z)\right).
\]
From the description of the conjugacy classes of involutions in $K$ and $G$ given in \cite{AS} we deduce that $z^G \cap K = z^K$ and thus
\[
{\rm fix}(z) = |\Delta| \cdot \frac{|z^G \cap K|}{|z^G|} = \frac{|C_G(z)|}{|C_K(z)|}.
\]
By reading off the relevant centralizer orders (see \cite[Tables 3.4.1, 3.5.1]{BG}, for example) we calculate that ${\rm fix}(z) = \frac{1}{2}q^{m}$ and thus $n=1$. In other words, $G=HK$ as required.
\end{proof}

\begin{rem}
The existence of a factorization $G = HK$ in this case with $H = q^m{:}(q^m-1)$ was first  established by Guralnick and Saxl in the proof of \cite[Theorem 3.1]{GS} (see p.139 and the first case considered in Step 7
of the proof). They proceed as follows. First let $J = {\rm Sp}_{2}(q^m) < G$. Then by arguing as in \cite{LPS} (see Case 3.2.1(d) on p.48) we deduce that $G = JK$ and $J \cap K = {\rm O}_{2}^{-}(q^m)$. Therefore, if $H = q^m{:}(q^m-1)$ is the Borel subgroup of $J$, then we must have $|H \cap K| = 2$ since $(|J \cap K|, |H|) = 2$, and we conclude that $G=HK$.
\end{rem}

\begin{prop}\label{p:case4}
Let $G$ be an almost simple group with socle $G_0 = {\rm Sp}_{4}(q)$, where $q$ is even, and let $H,K$ be subgroups satisfying the conditions in Case $4$ of Table $\ref{tab}$.
\begin{itemize}\addtolength{\itemsep}{0.2\baselineskip}
\item[{\rm (i)}] If $G=HK$, then $|H|$ is divisible by $q^2(q^2-1)$.
\item[{\rm (ii)}] If $G = {\rm Sp}_{4}(q)$ and $K = {\rm Sp}_{2}(q^2).2$, then $G=HK$ with $H = q^2{:}(q^2-1)<P_1$.
\end{itemize}
\end{prop}

\begin{proof}
By applying a graph automorphism of $G_0$, we may view Case 4 in Table \ref{tab} as a special version of Case 3 (with $m=2$) and so the result follows from Proposition \ref{p:case3}.
\end{proof}

\begin{prop}\label{p:case5}
Let $G$ be an almost simple group with socle $G_0 = {\rm PSp}_{4}(q)$, where $q$ is odd, and let $H,K$ be subgroups satisfying the conditions in Case $5$ of Table $\ref{tab}$.
\begin{itemize}\addtolength{\itemsep}{0.2\baselineskip}
\item[{\rm (i)}] If $G=HK$ then $|H|$ is divisible by $q^3(q^2-1)$.
\item[{\rm (ii)}] If $G = {\rm PGSp}_{4}(q)$ and $K = {\rm PGSp}_{2}(q^2).2$, then $G=HK$ with $H = q^{1+2}{:}(q^2-1)$ contained in $P_1$.
\end{itemize}
\end{prop}

\begin{proof}
In view of the isomorphism ${\rm PSp}_{4}(q) \cong \O_5(q)$, we can consider Case 5 in Table \ref{tab} as a special version of Case 7 (with $m=2$). Our analysis of Case 7 in Proposition \ref{p:case7} below includes the case $m=2$, and the result for Case 5 follows immediately.
\end{proof}

\subsection{Unitary groups}\label{ss:uni}

\begin{prop}\label{p:case6}
Let $G$ be an almost simple group with socle $G_0 = {\rm PSU}_{2m}(q)$, where $m \geqs 2$,
and let $H$ and $K$ be subgroups satisfying the conditions in Case $6$ of Table $\ref{tab}$.
\begin{itemize}\addtolength{\itemsep}{0.2\baselineskip}
\item[{\rm (i)}] If $G=HK$, then $|H|$ is divisible by $\frac{q^{2m}(q^{2m}-1)}{q+1}$.
\item[{\rm (ii)}] If $G = {\rm PGU}_{2m}(q)$ and $K = N_1$, then $G=HK$ with $H = q^{2m}{:}\frac{q^{2m}-1}{q+1}<P_m$.
\end{itemize}
\end{prop}

\begin{proof}
Let $V$ be the natural module for $G_0$ and fix a standard basis $\{e_1, \ldots, e_m,f_1, \ldots, f_m\}$ with respect to the defining unitary form on $V$. We have
\[
H \cap G_0 \leqs q^{m^2}{:}\left(\frac{q^{2m}-1}{(q+1)d}.m\right) < P,
\]
where $d=(2m,q+1)$ and $P$ is the stabilizer in $G_0$ of the totally isotropic space $\la e_1, \ldots,e_m\ra$. Here $q^{m^2}$ is the unipotent radical of $P$, which is elementary abelian. We may assume that $K$ is the stabilizer in $G$ of a nondegenerate $1$-space, so $K$ is a maximal subgroup of type ${\rm GU}_{2m-1}(q) \times {\rm GU}_{1}(q)$ and
\[
|G:K| = q^{2m-1}\frac{q^{2m}-1}{q+1}.
\]
The case $(m,q) = (3,2)$ can be handled using {\sc Magma}, so we may assume $(m,q) \ne (3,2)$.

Suppose $G=HK$. Then $|H|$ is divisible by $q^{2m-1}\frac{q^{2m}-1}{q+1}$ and thus $H$ contains an element $x$ of order $r$, a primitive prime divisor of $q^{2m}-1$. By considering the action of $\la x \ra$ on the unipotent radical of $P$ we deduce that $q^{2m}$ divides $|H|$ and part (i) follows.

Now let us turn to part (ii). If $m=2$ then $m^2=2m$ and the existence of the factorization in (ii) follows immediately from \cite[Proposition 5.2]{LX}. For the remainder let us assume $m \geqs 3$.

It will be convenient to work in the matrix group $G = {\rm GU}_{2m}(q)$, so we take $K = {\rm GU}_{2m-1}(q) \times {\rm GU}_{1}(q)$ and we consider a subgroup
\[
H \leqs q^{m^2}{:}\left((q^{2m}-1){:}m\right) < q^{m^2}{:}{\rm GL}_{m}(q^2) = U{:}L = P,
\]
where $U = q^{m^2}$ and $L = {\rm GL}_{m}(q^2)$ (once again, $P$ denotes the stabilizer in $G$ of $\la e_1, \ldots, e_m\ra$). Our goal is to construct a subgroup $H = q^{2m}{:}(q^{2m}-1)$ with $G=HK$, so that by passing to the quotient group $G/Z(G)$, we obtain a factorization of ${\rm PGU}_{2m}(q)$ of the required form.

Let $C = \la c \ra = C_{q^{2m}-1}$ be a Singer cycle in $L$ and let $W$ be a $(2m)$-dimensional irreducible summand for the action of $C$ on $U$. Let $N = N_L(C) = C{:}D$, where $D = \la \varphi \ra$ has order $m$. Set $H = W{:}C < P$.

\vs

\noindent \textbf{Claim.} \emph{We have $G=HK$.}

\vs

Let $X$ be the natural module for $L$ and let $\s:L \to L$ be the involutory automorphism induced by the action of the field automorphism $\l \mapsto \l^q$ on matrix entries. If $X$ affords the representation $\rho:L \to {\rm GL}_{m}(q^2)$, then we write $X^{(q)}$ for the space $X$ with $L$-action given by the representation $\s\rho$.

Consider the module $X \otimes X^{(q)}$ for $L$, which has dimension $m^2$ as a space over $\mathbb{F}_{q^2}$. This module is defined over $\mathbb{F}_q$ (see \cite[Theorem 5.1.13]{BHR}, for example) and we can identify $U$ with the $\mathbb{F}_q$-span of a suitable basis for $X \otimes X^{(q)}$. In terms of the above standard basis for $V$, each $u \in U$ can be expressed as a matrix of the form
\begin{equation}\label{e:u}
u = \left(\begin{array}{cc}
I_m & A \\
0 & I_m
\end{array}\right)
\end{equation}
for some matrix $A$ of size $m$ and it follows that $u$ has Jordan form $[J_2^{s},J_1^{2m-2s}]$ on the natural module for $G$, where $s={\rm rank}(A)$.

Now extend scalars and set $\bar{X} = X \otimes \mathbb{F}_{q^{2m}}$ and $\bar{U} = U \otimes \mathbb{F}_{q^{2m}}$, so we have
\[
U \subseteq X \otimes X^{(q)} \subseteq \bar{X} \otimes \bar{X}^{(q)} = \bar{U}.
\]
Fix a basis $\{x_1, \ldots, x_m\}$ for $\bar{X}$ comprising eigenvectors for $c$ (a generator for the Singer cycle $C<L$), say $x_i^c = \mu^{q^{2(i-1)}}x_i$ for each $i$, where $\mathbb{F}_{q^{2m}}^{\times} = \la \mu \ra$. We may assume that $\varphi$ cyclically permutes the $x_i$, sending $x_1$ to $x_2$, etc. Then
\[
\{x_i \otimes x_j \,:\, 1 \leqs i,j \leqs m\}
\]
is a basis for $\bar{U}$. In particular, if $u \in U$ and we write $u = \sum_{i,j}b_{ij}(x_i \otimes x_j)$ and $B = (b_{ij})$, then the rank of $B$ coincides with the rank of $A$ in \eqref{e:u}.

Set $J = {\rm GL}_{m}(q^2).2 = L.2 \leqs {\rm \Gamma L}_{m}(q^2)$ and observe that we may view $U$ and $\bar{U}$ as modules for $J$. Set $M = N_J(C) = C{:}E$, where $E = \la \sigma \ra$ has order $2m$ and $\s^2=\varphi$. Without loss of generality, we may assume that $\s$ sends $x_i \otimes x_j$ to $x_{j+1} \otimes x_i$ for all $i,j$ (reading subscripts mod $m$).

Consider $\bar{W} = W \otimes \mathbb{F}_{q^{2m}}$ as a subspace of $\bar{U}$. Now $W$ is absolutely irreducible as a module for $M$, so $\bar{W}$ is irreducible for $M$. It is straightforward to decompose $\bar{U}$ as a sum of irreducible $M$-modules. For convenience of notation, let us assume that $\bar{W}$ is the submodule generated by $x_1 \otimes x_3$, so
\[
\bar{W} = \la x_1 \otimes x_3, x_4 \otimes x_1, x_2 \otimes x_4, \ldots, x_3 \otimes x_m \ra
\]
(reading subscripts mod $m$). Now $\s$ has a $1$-dimensional fixed space on both $W$ and $\bar{W}$. On the latter space, this is clearly $\la w \ra$, where
\[
w  = x_1 \otimes x_3+ x_4 \otimes x_1+ x_2 \otimes x_4+ \cdots +x_3 \otimes x_m.
\]
We may assume that $w$ is contained in $W$ (the computations below do not change if $w$ is replaced by a scalar multiple), which implies that
\[
W =
\{\a(x_1 \otimes x_3)+ \a^q(x_4 \otimes x_1)+ \a^{q^2}(x_2 \otimes x_4)+ \cdots +\a^{q^{2m-1}}(x_3 \otimes x_m) \,:\, \a \in \mathbb{F}_{q^{2m}} \}.
\]
Suppose
\[
v = \a(x_1 \otimes x_3)+ \a^q(x_4 \otimes x_1)+ \a^{q^2}(x_2 \otimes x_4)+ \cdots +\a^{q^{2m-1}}(x_3 \otimes x_m) = \sum_{i,j}b_{ij}(x_i \otimes x_j)
\]
is a nontrivial element of $W$ and set $B = (b_{ij})$ and $\ell = \frac{q^{2m}-1}{q+1}$. Then it is straightforward to check that
\[
{\rm rank}(B) = \left\{\begin{array}{ll}
m-1 & \mbox{if $\a^{\ell}=(-1)^m$} \\
m & \mbox{otherwise}
\end{array}\right.
\]
and thus $W$ contains $\ell$ elements with Jordan form $[J_2^{m-1},J_1^2]$ and the remaining nontrivial elements have Jordan form $[J_2^m]$.

Let $\Delta = G/K$ and let $n$ be the number of orbits of $H$ on $\Delta$. Note that we may identify $\Delta$ with the set of $1$-dimensional nondegenerate subspaces of the natural module for $G$. Also observe that $Z(G) = C_{q+1}$ is the kernel of the action.

Since $H \cap K = (W \cap K) \times Z(G)$ and elements of type $[J_2^m]$ have no fixed points on $\Delta$, it follows that
\[
n = \frac{1}{|H|}\sum_{h \in H}{\rm fix}(h) = \frac{1}{|H|}\left((q+1)|\Delta| + \ell(q+1)\cdot {\rm fix}(z)\right),
\]
where $z \in W \cap K$ has Jordan form $[J_{2}^{m-1},J_1^2]$. Now $z^G \cap K = z^K$ and we calculate that
\[
{\rm fix}(z) = \frac{|C_G(z)|}{|C_K(z)|} = q^{2m-1}(q-1).
\]
It is now routine to check that this gives $n=1$ and thus $G=HK$.
\end{proof}

\subsection{Orthogonal groups}\label{ss:ort}

We are now ready to complete the proof of Theorem \ref{t:main} by handling Cases 7, 8 and 9 in Table \ref{tab}.

First we consider Case 7. In the statement of the following proposition, it is convenient to allow the case $m=2$, which is excluded in Table \ref{tab} (cf. Proposition \ref{p:case5}). In part (ii), we use $N_{1}^{-}$ to denote the stabilizer in ${\rm SO}_{2m+1}(q)$ of a nondegenerate $1$-space $W$ such that $W^{\perp}$ is a minus-type orthogonal space.

\begin{prop}\label{p:case7}
Let $G$ be an almost simple group with socle $G_0 = \O_{2m+1}(q)$, where $m \geqs 2$ and $q$ is odd. Let $H,K$ be subgroups satisfying the conditions in Case $7$ of Table $\ref{tab}$.
\begin{itemize}\addtolength{\itemsep}{0.2\baselineskip}
\item[{\rm (i)}] If $G=HK$, then $|H|$ is divisible by $\frac{1}{e}q^{\frac{1}{2}m(m+1)}(q^m-1)$, where
\[
e=\left\{
 \begin{array}{ll}
2 & \mbox{if $q^m \equiv 3 \imod{4}$} \\
1 & \mbox{otherwise.}
\end{array}\right.
\]
\item[{\rm (ii)}] If $G = {\rm SO}_{2m+1}(q)$ and $K = N_1^{-}$, then $G=HK$ with
\[
H = (q^{\frac{1}{2}m(m-1)}.q^m){:}\frac{q^m-1}{e}<P_m.
\]
\end{itemize}
\end{prop}

\begin{proof}
Let $V$ be the natural module for $G_0$ and fix a standard basis
\[
\{e_1, \ldots, e_m,f_1, \ldots, f_m, v\}.
\]
Let $P$ be the stabilizer in $G_0$ of the maximal totally singular subspace
$\la e_1, \ldots, e_m\ra$ and write $P = U{:}L$, where the unipotent radical $U$ is nonabelian of order $q^{\frac{1}{2}m(m+1)}$ and $L = \frac{1}{2}{\rm GL}_{m}(q)$ is a Levi factor. Here $Z(U) = q^{\frac{1}{2}m(m-1)}$ is the unique minimal normal subgroup of $P$ and we may identify $U/Z(U)$ with the natural module for ${\rm GL}_{m}(q)$. Write $q=p^f$, where $p$ is a prime.

Now
\[
H \cap G_0 \leqs (q^{\frac{1}{2}m(m-1)}.q^m){:}\left(\frac{q^m-1}{2}.m\right) < P
\]
and we may assume that $K$ is the stabilizer in $G$ of the $1$-space $\la e_m+\mu f_m\ra$, where $\mu \in \mathbb{F}_q$ is a non-square. In other words, $K$ is a maximal subgroup of $G$ of type ${\rm O}_{2m}^{-}(q)$ and we note that
$|G:K| = \frac{1}{2}q^m(q^m-1)$.

As in the proof of \cite[Proposition 5.5]{LX}, define the following linear maps on $V$ for $1\leqs i < j \leqs m$ and $\l \in \mathbb{F}_q$, where
\[
y_{i,j}(\l): \;\; f_i \mapsto f_i+\l e_j,\;\; f_j \mapsto f_j -\l e_i
\]
and $y_{i,j}(\l)$ fixes the remaining basis elements. It is straightforward to check that each $y_{i,j}(\l)$ is contained in $Z(U)$. Therefore, since $|Z(U)| = q^{\frac{1}{2}m(m-1)}$, we conclude that
\[
Z(U) = \la y_{i,j}(\l) \,:\, 1\leqs i<j \leqs m,\; \l \in \mathbb{F}_q \ra
\]
and thus
\[
Z(U) \cap K = \la y_{i,j}(\l) \,:\, 1\leqs i<j < m,\; \l \in \mathbb{F}_q \ra
\]
is elementary abelian of order $q^{\frac{1}{2}(m-1)(m-2)}$.

\vs

\noindent \textbf{Claim.} \emph{We have $G = HK$ only if $U \leqs H$.}

\vs

Seeking a contradiction, suppose $G=HK$ and $W = H \cap U<U$. Notice that if
${\rm SO}_{2m+1}(q) \not\leqs G$, then by taking $G_1 = \la {\rm SO}_{2m+1}(q), G \ra$ we can construct a factorization $G_1 = HK_1$ with $K_1 = N_{G_1}(K)$. Therefore, without loss of generality, we may assume that $G$ contains ${\rm SO}_{2m+1}(q)$, so $G = {\rm SO}_{2m+1}(q).f'$ for some divisor $f'$ of $f$.

First assume $W=Z(U)$. Then without loss of generality, we may assume that $H = W{:}(C.m.f')$, where $C<{\rm GL}_{m}(q)$ is a Singer cycle. Since $Z(U) \cap K \leqs H \cap K$, it follows that $q^{\frac{1}{2}(m-1)(m-2)}$ divides
\[
|H \cap K| = \frac{|H||K|}{|G|} = 2mf'q^{\frac{1}{2}m(m-3)}
\]
and thus $q$ divides $2mf'$. Let $J = W{:}C$, which is a normal subgroup of $H$ of index $mf'$. Then $J \cap K$ is normal in $H \cap K$ and the proof of \cite[Proposition 5.5]{LX} implies that $J \cap K = W \cap K$. Therefore, we can write
\[
H \cap K = (J \cap K).R,
\]
where $R \leqs C.m.f'$ has order $\frac{2mf'}{q}$. Since $|R \cap C|=1$, we have
\[
R \cong RC/C \leqs (C.m.f')/C \cong C_m.C_{f'}
\]
and thus $|R|=\frac{2mf'}{q}$ divides $mf'$. But this is absurd since $q$ is odd.

Next suppose $W<Z(U)$. Then $G = \tilde{H}K$, where $\tilde{H} = Z(U)H$ and $\tilde{H} \cap U = Z(U)$. But we have just ruled out the existence of such a factorization, so the case $W<Z(U)$ does not arise. 

To complete the proof of the claim, we may assume $W \cap Z(U) < W$, in which case $WZ(U)/Z(U) \leqs U/Z(U)$ is nontrivial. Write $H = W{:}D$, where $|D|$ is divisible by $\frac{q^m-1}{2}$, and note that
\[
HZ(U)/Z(U) = (WZ(U)/Z(U)){:}D \leqs (U/Z(U)){:}D.
\]

Suppose $q^m-1$ is divisible by a primitive prime divisor $r$. Then $D$ contains an element $x$ of order $r$ and $\la x \ra$ does not normalize a proper nontrivial subgroup of $U/Z(U)$. Therefore, $WZ(U)/Z(U) = U/Z(U)$ and thus $WZ(U)=U$. In particular, $W \cap Z(U)$ is a normal subgroup of $U$. If $W \cap Z(U)=1$ then $W \cong U/Z(U)$ is abelian and thus $U \cong W \times Z(U)$ is abelian, which is a contradiction. Therefore, $W \cap Z(U)$ is a nontrivial normal subgroup of $U$, so it must contain $Z(U)$, which is the unique minimal normal subgroup of $U$. It follows that $Z(U) \leqs W \cap Z(U)$, so $W=U$ and once again we have reached a contradiction.

Finally, let us assume $q^m-1$ does not have a primitive prime divisor, so Zsigmondy's theorem \cite{Zsig} implies that $m=2$ and $q$ is a Mersenne prime. In particular, $G= {\rm SO}_{5}(q)$, $|U|=q^3$, $|Z(U)|=q$ and $|W|=q^2$ (recall that $|H|$ is divisible by $|G:K| = \frac{1}{2}q^2(q^2-1)$). Let us also note that $D \leqs (q^2-1){:}2$. The case $q=3$ can be handled directly using {\sc Magma} \cite{magma}, so we may assume $q \geqs 7$. Then $D$ contains an element $x$ of order $\frac{q^2-1}{4}$ and $\la x \ra$  does not normalize a proper nontrivial subgroup of $U/Z(U)$. Therefore, $WZ(U)/Z(U) = U/Z(U)$ and we reach a contradiction by repeating the argument in the previous paragraph.

\vs

With the claim in hand, let $G$ be an arbitrary almost simple group with socle $G_0$ and a factorization $G=HK$. Since $U \leqs H$ and $\frac{1}{2}q^m(q^m-1)$ divides $|H|$, it follows that $|H|$ is divisible by $\frac{1}{2}q^{\frac{1}{2}m(m+1)}(q^m-1)$.

Suppose $|H| = \frac{1}{2}q^{\frac{1}{2}m(m+1)}(q^m-1)$. As before, without loss of generality we can assume that $G = {\rm SO}_{2m+1}(q).f'$, where $f'$ divides $f$. Then we may write
\[
H = U{:}\left(\frac{q^m-1}{2ab}.a\right)\!.b,
\]
where $a$ divides $m$ and $b$ divides $f'$. Let $C<{\rm GL}_{m}(q)$ be a Singer cycle and let $D$ be the unique index-two subgroup of $C$.

First assume $a=b=1$, so $H = U{:}D$ and we define $J = U{:}C$. Then \cite[Proposition 5.7]{LX} implies that $G=JK$, so $J$ acts transitively on $\Delta = G/K$. In turn, $U$ acts $\frac{1}{2}$-transitively on $\Delta$ (that is, all the orbits of $U$ have the same size, namely $q^m$) and the orbits of $U$ form a block system $\mathcal{B}$ for the action of $J$, with $|\mathcal{B}| = \frac{1}{2}(q^m-1)$. Now $J/U \cong C$ is transitive on $\mathcal{B}$ and thus the unique involution $z \in C$ is in the kernel of this action. It follows that $D$ acts transitively on $\mathcal{B}$, whence $G = HK$ if and only if $z \not\in D$. That is, $G=HK$ if and only if $|D|$ is odd, which is equivalent to the condition $q^m \equiv 3 \imod{4}$.

Now assume $(a,b) \ne (1,1)$ and write $H = U{:}E$, where $E=\left(\frac{q^m-1}{2ab}.a\right)\!.b$. As above, the orbits of $U$ on $\Delta=G/K$ form a block system $\mathcal{B}$ with $|\mathcal{B}| = \frac{1}{2}(q^m-1)$. Since $C$ acts transitively on $\mathcal{B}$, it follows that the unique involution in $C$ is contained in the kernel of this action.

\vspace{2mm}

\noindent \textbf{Claim.} \emph{If $q^m \equiv 1 \imod{4}$, then $\frac{q^m-1}{2ab}$ is even.}

\vspace{2mm}

Write $fm=2^k\ell$ with $\ell$ odd and set $r=p^{\ell}$.
Then $q^m=p^{fm}=r^{2^k}$.
Since $a$ divides $m$ and $b$ divides $f$, the 2-part of $ab$ divides $2^k$, which is the 2-part of $fm$. Factorize $q^m-1$ as follows
\[
q^m-1=r^{2^k}-1=(r^{2^{k-1}}+1)(r^{2^{k-2}}+1)\cdots(r+1)(r-1).
\]
Since each factor $r^{2^i}+1$ is even and $(r+1)(r-1)$ is divisible by 8, we conclude that $q^m-1$ is divisible by $4ab$ and so $\frac{q^m-1}{2ab}$ is even as claimed.

\vspace{2mm}

It follows that if $q^m \equiv 1 \imod{4}$ then $E$ contains the involution in $C$ and thus $E$ is intransitive on $\mathcal{B}$. In particular, $G \ne HK$. An entirely similar argument shows that if
$q^m \equiv 1 \imod{4}$ then $|H|$ is divisible by $q^{\frac{1}{2}m(m+1)}(q^m-1)$ and we have now established part (i).

Finally, let us turn to part (ii), so $G = {\rm SO}_{2m+1}(q)$. If $q \equiv 1 \imod{4}$ then the result follows immediately from \cite[Proposition 5.7]{LX}. Similarly, if $q^m \equiv 3 \imod{4}$ then the above argument (the case $a=b=1$) shows that $G=HK$ with $H = U{:}D$ and $K = N_{1}^{-}$. This completes the proof of the proposition.
\end{proof}

Next we handle Case 8 in Table \ref{tab}. Note that we include the case $m=4$, so this also covers Case 9 in Table \ref{tab}. In particular, the following proposition completes the proof of Theorem \ref{t:main}. In part (iii), $N_1$ denotes the stabilizer of a nondegenerate $1$-space.

\begin{prop}\label{p:case8}
Let $G$ be an almost simple group with socle $G_0 = {\rm P\O}_{2m}^{+}(q)$, where $m \geqs 4$, and let $H$ and $K$ be subgroups satisfying the conditions in Case $8$ of Table $\ref{tab}$.
\begin{itemize}\addtolength{\itemsep}{0.2\baselineskip}
\item[{\rm (i)}] If $G=HK$, then $|H|$ is divisible by $\frac{q^{m}(q^m-1)}{(2,q-1)}$.
\item[{\rm (ii)}] If $q$ is even, $G = \O_{2m}^{+}(q)$ and $K = {\rm Sp}_{2m-2}(q)$, then $G=HK$ with $H = q^m{:}(q^m-1)$.
\item[{\rm (iii)}] If $q$ is odd, $G = {\rm PSO}_{2m}^{+}(q)$ and $K = N_{1}$, then $G=HK$ with $H = q^m{:}\frac{q^m-1}{2}$.
\end{itemize}
In parts (ii) and (iii), the given subgroup $H$ is contained in $P_m$.
\end{prop}

\begin{proof}
To begin with, let us assume $q$ is even. Let $V$ be the natural module for $G_0$ and fix a standard basis $\mathcal{B}=\{e_1, \ldots, e_m,f_1, \ldots, f_m\}$. Let $P = U{:}L$ be the stabilizer in $G_0$ of the totally singular subspace $\la e_1, \ldots, e_m\ra$, where the unipotent radical $U$ is elementary abelian of order $q^{\frac{1}{2}m(m-1)}$ and $L = {\rm GL}_{m}(q)$ is a Levi factor. Then
\[
H \cap G_0 \leqs q^{\frac{1}{2}m(m-1)}{:}\left((q^m-1){:}m\right)<P
\]
and we may assume that $K$ is maximal among subgroups of $G$ with $K \cap G_0  = {\rm Sp}_{2m-2}(q)$. Note that $|G:K|$ is divisible by $q^{m-1}(q^m-1)$. We can use {\sc Magma} to handle the case $G_0 = \O_{12}^{+}(2)$, so we may assume $(m,q) \ne (6,2)$.

If $G=HK$ then $|H|$ is divisible by $q^{m-1}(q^m-1)$ and by repeating the argument in the proof of Proposition \ref{p:case3}, we deduce that $q^m(q^m-1)$ divides $|H|$. This gives part (i).

Next let us turn to (ii). Set $G = \Omega_{2m}^{+}(q)$. Let $C = \la c \ra$ be a Singer cycle in $L$ and set $N = N_L(C) = C{:}D$, where $D = \la \varphi \ra$ has order $m$.
As a module for $L$, we may identify $U$ with $\L^2(X)$, where $X$ is the natural module for $L$. In terms of the basis $\mathcal{B}$, we can express each $u \in U$ as in \eqref{e:u} for some matrix $A$ of size $m$ and we see that $u$ has Jordan form $[J_2^s,J_1^{2m-2s}]$ on $V$, where $s={\rm rank}(A)$. Moreover, one checks that every nontrivial element in $U$ is an involution of type $a$ (with respect to \cite{AS}), so $u$ is $G$-conjugate to $a_s$.

Write $\bar{X} = X \otimes \mathbb{F}_{q^m}$ and $\bar{U} = \L^2(\bar{X})$. The action of $c$ on $\bar{X}$ is diagonalizable and we may choose a basis $\{x_1, \ldots, x_m\}$ for $\bar{X}$ such that $x_i^c = \mu^{q^{i-1}}x_i$ for $i=1, \ldots, m$, where $\mathbb{F}_{q^m}^{\times} = \la \mu \ra$. Then
\[
\{x_i \wedge x_j \,:\, 1\leqs i < j \leqs m\}
\]
is a basis for $\bar{U}$ and we may assume that $\varphi$ cyclically permutes the $x_i$. In particular, if $u \in U$ and we write $u = \sum_{i,j}b_{ij}(x_i \wedge x_j)$, then the corresponding alternating matrix $B = (b_{ij})$ of size $m$ has the same rank as the matrix $A$ in \eqref{e:u}. Therefore, the rank of $B$ determines the $G$-class of $u$.

Consider the action of $C$ on $U$ and let $W$ be an irreducible submodule of dimension $m$. Then $W$ is absolutely irreducible as a module for $N$ and thus
$\bar{W} = W \otimes \mathbb{F}_{q^m}$ is an irreducible $N$-module. It is straightforward to decompose $\bar{U}$ as a direct sum of irreducible modules for $N$ and we may assume that $W$ has been chosen so that
\[
\bar{W} = \la x_1 \wedge x_{2}, x_2 \wedge x_{3}, \ldots, x_m \wedge x_{1} \ra.
\]
Set
\[
w = x_1 \wedge x_2 + x_2 \wedge x_3 + \cdots + x_{m-1}\wedge x_m + x_m \wedge x_1 \in \bar{W}
\]
and $H = W{:}C = q^m{:}(q^m-1)$.

\vs

\noindent \textbf{Claim.} \emph{We have $G=HK$.}

\vs

To see this, observe that $\varphi \in N$ has a $1$-dimensional fixed space on both $W$ and $\bar{W}$. Clearly, $\la w \ra$ is the fixed point space of $\varphi$ on $\bar{W}$, so $W$ must contain a nonzero scalar multiple of $w$. Without loss of generality, we may assume that $w \in W$, so $W$ is generated by $w$ as a $C$-module and we deduce that 
\[
W = \{\a(x_1 \wedge x_2) + \a^q(x_2 \wedge x_3) + \cdots + \a^{q^{m-2}}(x_{m-1}\wedge x_m) + \a^{q^{m-1}}(x_m \wedge x_1) \,:\, \a \in \mathbb{F}_{q^m}\}.
\]

If $m$ is odd and $z \in W$ is nontrivial, then a straightforward calculation shows that ${\rm rank}(z) = m-1$ and thus $z$ is an involution in $G$ of type $a_{m-1}$. By applying the orbit counting lemma with respect to the action of $H$ on $\Delta = G/K$, noting that we have $|\Delta| = q^{m-1}(q^m-1)$ and ${\rm fix}(h) = 0$ for all $h \in H \setminus W$ (see the proof of \cite[Proposition 5.9]{LX}), we calculate that $H$ has
\[
n = \frac{1}{|H|}\left(|\Delta| + (q^m-1)\cdot {\rm fix}(z)\right)
\]
orbits on $\Delta$. Now $z^G \cap K = z^K$ and we deduce that
\[
{\rm fix}(z) = \frac{|C_G(z)|}{|C_K(z)|} = q^{m-1}(q-1),
\]
which yields $n=1$. Therefore, $G = HK$ as claimed.

Now assume $m$ is even and fix a nontrivial element
\[
z = \a(x_1 \wedge x_2) + \a^q(x_2 \wedge x_3) + \cdots + \a^{q^{m-2}}(x_{m-1}\wedge x_m) + \a^{q^{m-1}}(x_m \wedge x_1) \in W.
\]
Set $\ell = \frac{q^m-1}{q+1}$. A calculation shows that
\[
{\rm rank}(z) = \left\{\begin{array}{ll}
m-2 & \mbox{if $\a^{\ell} =1$} \\
m & \mbox{otherwise.}
\end{array}\right.
\]
Therefore, $W$ contains $\ell$ involutions of type $a_{m-2}$, and the remaining involutions in $W$ are of type $a_m$. Since the latter involutions have no fixed points on $\Delta=G/K$, it follows that $H$ has
\[
n = \frac{1}{|H|}\left(|\Delta| + \ell\cdot {\rm fix}(z)\right)
\]
orbits on $\Delta$, where $z$ is an involution of type $a_{m-2}$. It is straightforward to check that
\[
{\rm fix}(z) = \frac{|C_G(z)|}{|C_K(z)|} = q^{m-1}(q^2-1)
\]
and this gives $n=1$ as required.

\vs

To complete the proof of the proposition, let us assume $q$ is odd. Here we take $K$ to be maximal among subgroups of $G$ with $K \cap G_0 \trianglerighteqslant \O_{2m-1}(q)$, so $|G:K|$ is divisible by $\frac{1}{2}q^{m-1}(q^m-1)$. In particular, if $G=HK$ then $|H|$ is divisible by $\frac{1}{2}q^{m-1}(q^m-1)$ and by considering an element of order $r$, a primitive prime divisor of $q^m-1$, we quickly deduce that $q^m$ divides $|H|$. Therefore, $\frac{1}{2}q^{m}(q^m-1)$ divides $|H|$ as in part (i).

Finally, let us turn to part (iii). Here it will be convenient to work in the matrix group $G = {\rm SO}_{2m}^{+}(q)$ and we will establish the existence of a factorization $G=HK$ with $H = q^m{:}(q^m-1)$. Then by taking quotients, we will obtain the desired factorization of $G/Z(G) = {\rm PSO}_{2m}^{+}(q)$.

As before, let $\{e_1, \ldots, e_m, f_1, \ldots, f_m\}$ be a standard basis for the natural module $V$ and let $P=U{:}L$ be the stabilizer in $G$ of the subspace $\la e_1, \ldots, e_m \ra$, where $U$ is elementary abelian of order $q^{\frac{1}{2}m(m-1)}$ and $L = {\rm GL}_{m}(q)$. Each $u \in U$ can be expressed as a matrix as in \eqref{e:u} and we deduce that $u$ has Jordan form $[J_2^s,J_{1}^{2m-2s}]$ on $V$, where $s = {\rm rank}(A)$. Let $C = \la c \ra < L$ be a Singer cycle and let $W \subseteq U$ be an irreducible summand of dimension $m$ with respect to the action of $C$ on $U$.

Let $H = W{:}C = q^m{:}(q^m-1)$ and let $\Delta = G/K$, where $K = {\rm SO}_{2m-1}(q) \times C_2$ is the stabilizer of a nondegenerate $1$-space. We claim that $G=HK$. As before, we will use the orbit counting lemma to show that $n=1$, where $n$ is the number of orbits of $H$ on $\Delta$. Note that $Z(G) = C_2$ is the kernel of the action, which contains the element of order $2$ in $C$. By the proof of \cite[Proposition 5.9]{LX}, we have $H \cap K = (W \cap K) \times Z(G)$.

First assume $m$ is odd. By repeating the argument in the $q$ even case, we deduce that every nontrivial element in $W$ has Jordan form $[J_{2}^{m-1},J_1^2]$. Now $G$ has a unique conjugacy class of such elements and thus
\[
n = \frac{1}{|H|}\sum_{h \in H}{\rm fix}(h) = \frac{1}{|H|}\left(2|\Delta| + 2(q^m-1)\cdot {\rm fix}(z)\right),
\]
where $z \in W \cap K$ is nontrivial. Since ${\rm SO}_{2m-1}(q)$ has a unique class of unipotent elements with Jordan form $[J_{2}^{m-1},J_1]$, we deduce that $z^G \cap K = z^K$ and so we get
\[
{\rm fix}(z) = |\Delta| \cdot \frac{|z^G \cap K|}{|z^G|} = \frac{|C_G(z)|}{|C_K(z)|}.
\]
By reading off the relevant centralizer orders (see \cite[Section 3.5.3]{BG}), we deduce that
\[
|C_G(z)| = q^{\frac{1}{2}m^2+\frac{1}{2}m-1}|{\rm Sp}_{m-1}(q)||{\rm SO}_{2}^{+}(q)|,\; |C_K(z)| = 2q^{\frac{1}{2}m^2-\frac{1}{2}m}|{\rm Sp}_{m-1}(q)|
\]
and thus
\[
{\rm fix}(z) = \frac{1}{2}q^{m-1}(q-1).
\]
One now checks that this gives $n=1$, as required.

A similar argument applies when $m$ is even. Proceeding as in the $q$ even case, we deduce that $\ell = \frac{q^m-1}{q+1}$ elements in $W$ have Jordan form $[J_2^{m-2},J_1^4]$ on $V$, and the remaining nontrivial elements have Jordan form $[J_2^m]$. Since none of the latter elements fix a nondegenerate $1$-space, it follows that
\[
n = \frac{1}{|H|}\sum_{h \in H}{\rm fix}(h) = \frac{1}{|H|}\left(2|\Delta| + 2\ell \cdot {\rm fix}(z)\right),
\]
where $z \in W \cap K$ is nontrivial. As above, we have $z^G \cap K = z^K$ and we calculate that
\[
{\rm fix}(z) = \frac{|C_G(z)|}{|C_K(z)|} = \frac{1}{2}q^{m-1}(q^2-1).
\]
This gives $n=1$ and the proof of the proposition is complete.
\end{proof}

\vs

This completes the proof of Theorem \ref{t:main}.

\section{Proof of Theorem \ref{t:main2}}\label{s:main2}

In this section we prove Theorem \ref{t:main2}, which gives the minimal order
$\ell(G_0)$ of a solvable factor of any almost simple group $G$ with socle $G_0$.

\begin{prop}\label{p:ell1}
The conclusion to Theorem \ref{t:main2} holds if $G_0$ is non-classical.
\end{prop}

\begin{proof}
As noted in Section \ref{s:intro}, we have $\ell(G_0)=0$ if $G_0$ is an exceptional group of Lie type. If $G_0 = A_n$ then every core-free subgroup of $G$ has index at least $n$, so every solvable factor has order at least $n$. But if $G = S_n$ and $H = \la g \ra$ with $g$ an $n$-cycle, then $G=HK$ with $K = S_{n-1}$, whence $\ell(G_0) = n$. Finally, if $G_0$ is a sporadic group then the relevant factorizations are recorded in \cite[Proposition 4.4]{LX} and we can read off $\ell(G_0)$.
\end{proof}

To complete the proof of Theorem \ref{t:main2}, we need to handle the factorizations $G=HK$ of almost simple classical groups. Recall from Section \ref{s:intro} that such a factorization is of Type I, II or III, according to the following definition:
\begin{itemize}\addtolength{\itemsep}{0.2\baselineskip}
\item[{\rm I.}] $(G_0, H \cap G_0, K\cap G_0)$ belongs to one of 9 infinite families listed in \cite[Table 1.1]{LX};
\item[{\rm II.}] $(G_0, H \cap G_0, K\cap G_0)$ is one of 28 ``sporadic" cases recorded in \cite[Table 1.2]{LX};
\item[{\rm III.}] Both $H$ and $K$ are solvable and $(G,H,K)$ is recorded in \cite[Proposition 4.1]{LX}.
\end{itemize}

Since Theorem \ref{t:main} gives a sharp lower bound on $|H|$ for each Type I factorization, it remains to consider the Type II and III factorizations. For $j = 1, \ldots, 28$, a \emph{Type II.j factorization} satisfies the conditions of Case j in \cite[Table 1.2]{LX}.

\begin{lem}\label{l:ell2}
Let $G$ be an almost simple classical group with socle $G_0$ and suppose $G=HK$ is a Type II.j factorization of $G$. Then there exists an integer $\ell$ and an almost simple group $G_1$ with socle $G_0$ such that
\begin{itemize}\addtolength{\itemsep}{0.2\baselineskip}
\item[{\rm (i)}] $G_1 = H_1K_1$ is a Type II.j factorization of $G_1$; and
\item[{\rm (ii)}] $|H_1| = \ell$ divides $|H|$,
\end{itemize}
where $(G_1,H_1,K_1,\ell)$ is listed in Table $\ref{tab3}$.
\end{lem}

\begin{proof}
We inspect each case in \cite[Table 1.2]{LX} in turn, using {\sc Magma} \cite{magma} to determine the minimal order of a solvable factor and to verify the existence of the factorizations presented in Table \ref{tab3} (note that in Case 20 of Table \ref{tab3}, $G_1 = {\rm PSU}_{4}(3).2$ is a subgroup of ${\rm PGU}_{4}(3)$).
\end{proof}

\begin{table}
\[
\begin{array}{clllc} \hline
\mbox{Case} & G_1 & H_1 & K_1  & \ell \\ \hline
1  & {\rm PSL}_{2}(11) & C_{11} & A_5 & 11 \\
2  & {\rm P\Gamma L}_{2}(16) & D_{34}.4 & 2.S_5 & 136 \\
3  & {\rm PSL}_{2}(19) & 19{:}9 & A_5 & 171\\
4  & {\rm PSL}_{2}(29) & 29{:}7 & A_5 & 203 \\
5  & {\rm PSL}_{2}(59) & 59{:}29 & A_5 & 1711 \\
6 & {\rm PGL}_{4}(3) & 2^4{:}5 & 3^3{:}{\rm GL}_{3}(3) & 80 \\
7  & {\rm PGL}_{4}(3) & (3^3{:}13{:}3).2 & ((4 \times {\rm PSL}_{2}(9)){:}2).2 & 2106 \\
8  & {\rm P \Gamma L}_{4}(4) & (2^6{:}63{:}3).2 & ((5 \times {\rm PSL}_{2}(16)){:}2).2 & 24192 \\
9  & {\rm PSL}_{5}(2) & 31{:}5 & 2^6{:}(S_3 \times {\rm PSL}_{3}(2)) & 155 \\
10  & {\rm PSp}_{4}(3) & 3_{+}^{1+2} & 2^4{:}A_5 & 27 \\
11  & {\rm PSp}_{4}(3) & 3_{+}^{1+2} & 2^4{:}A_5 & 27 \\
12  & {\rm PSp}_{4}(5) & 5_{+}^{1+2}{:}2.A_4 & {\rm PSL}_2(5^2).2 & 3000 \\
13  & {\rm PSp}_{4}(7) & 7_{+}^{1+2}{:}2.S_4 & {\rm PSL}_2(7^2).2 & 16464 \\
14  & {\rm PSp}_{4}(11) & 11_{+}^{1+2}{:}10.A_4 & {\rm PSL}_2(11^2).2 & 159720 \\
15  & {\rm PSp}_{4}(23) & 23_{+}^{1+2}{:}20.S_4 & {\rm PSL}_2(23^2).2 & 6424176 \\
16  & {\rm Sp}_{6}(2) & 3_{+}^{1+2}{:}Q_8 & S_8 & 216 \\
17  & {\rm PGSp}_{6}(3) & (3^{1+4}_{+}{:}2^{1+4}.D_{10}).2 & ({\rm PSL}_{2}(27){:}3).2 & 155520 \\
18  & {\rm PSU}_{3}(3) & 3^{1+2}_{+}{:}8 & {\rm PSL}_{2}(7) & 216 \\
19  & {\rm PSU}_{3}(5) & 5^{1+2}_{+}{:}8 & A_7 & 1000 \\
20  & {\rm PSU}_{4}(3).2 & 3^4{:}2 & {\rm PSL}_{3}(4).2 & 162 \\
21  & {\rm PSU}_{4}(8).3 & (513{:}3).3 & (2^{12}.{\rm SL}_{2}(64).7).3 & 4617 \\
22  & \O_7(3) & 3^3{:}2^4.5 & G_2(3) & 19440 \\
23  & \O_7(3) & 3^{3+3}{:}13 & {\rm Sp}_{6}(2) & 9477 \\
24  & \O_9(3) & 3^{6+4}{:}2^{1+4}.5 & \O_{8}^{-}(3).2 & 9447840 \\
25  & \O_{8}^{+}(2) & 2^2{:}15.4 & {\rm Sp}_{6}(2) & 240 \\
26  & \O_{8}^{+}(2) & 2^6{:}15
& A_9 & 960 \\
27  & {\rm PSO}_{8}^{+}(3) & 3^4.4.D_{10} & {\rm SO}_7(3) & 3240 \\
28  & {\rm P\O}_{8}^{+}(3) & 3^6{:}(3^3{:}13) & \O_8^{+}(2) & 255879 \\ \hline
\end{array}
\]
\caption{Type II factorizations $G_1 = H_1K_1$ with $|H_1|=\ell$ in Lemma  \ref{l:ell2}}
\label{tab3}
\end{table}

Finally, we turn to the Type III factorizations, which are described in \cite[Proposition 4.1]{LX}. The latter result partitions the possible factorizations $G=HK$ into $12$ cases and we will refer to \emph{Type III.j factorizations} accordingly, for $j = 0,1, \ldots, 11$. The Type III.0 factorizations correspond to the infinite family described in \cite[Proposition 4.1(a)]{LX}, which we have already handled in our analysis of Type I.1 factorizations with $n=2$ (see Remark \ref{r:1}(a)). The factorizations of Type III.j with $j=1, \ldots, 11$ coincide with the cases recorded in row j of \cite[Table 4.1]{LX}.

\begin{lem}\label{l:ell3}
Let $G$ be an almost simple classical group with socle $G_0$ and suppose $G=HK$ is a Type III.j factorization of $G$. Then there exists an integer $\ell$ and an almost simple group $G_1$ with socle $G_0$ such that
\begin{itemize}\addtolength{\itemsep}{0.2\baselineskip}
\item[{\rm (i)}] $G_1 = H_1K_1$ is a Type III.j factorization of $G_1$; and
\item[{\rm (ii)}] $|H_1|=\ell$ divides $|H|$,
\end{itemize}
where $(G_1,H_1,K_1,\ell)$ is listed in Table $\ref{tab4}$.
\end{lem}

\begin{proof}
As noted above, the Type III.0 factorizations have already been handled in our treatment of Type I.1 factorizations in Theorem \ref{t:main}, so we may assume $j = 1, \ldots, 11$. Here the possibilities for $G$, $H$ and $K$ are listed in \cite[Table 4.1]{LX} and the desired result quickly follows.
\end{proof}

\begin{table}
\[
\begin{array}{clllc} \hline
{\rm Case} & G_1 & H_1 & K_1  & \ell \\ \hline
0 & {\rm PGL}_{2}(q) & C_{q+1} & P_1 & q+1 \\
1 & {\rm PSL}_{2}(7) & C_7 & S_4 & 7 \\
2 & {\rm PSL}_{2}(11) & 11{:}5 & A_4 & 55 \\
3 & {\rm PSL}_{2}(23) & 23{:}11 & S_4 & 253 \\
4 & {\rm PSL}_{3}(3) & C_{13} & 3^2{:}2.S_4 & 13 \\
5 & {\rm PSL}_{3}(3) & 13{:}3 & {\rm A\Gamma L}_{1}(9) & 39 \\
6 & {\rm PSL}_{3}(4).S_3 & 7{:}3.S_3 & 2^4{:}(3 \times D_{10}).2 & 126 \\
7 & {\rm PSL}_{3}(8).3 & 73{:}9 & 2^{3+6}{:}7^2{:}3 & 657 \\
8 & {\rm PSU}_{3}(8).3^2 & 57{:}9 & 2^{3+6}{:}(63{:}3) & 513 \\
9 & {\rm PSU}_{4}(2) & 2^4{:}5 & 3_{+}^{1+2}{:}2.A_4 & 80 \\
10 & {\rm PSU}_{4}(2) & 2^4{:}D_{10} & 3_{+}^{1+2}{:}2.A_4 & 160 \\
11 & {\rm PSU}_{4}(2).2 & 2^4{:}5{:}4 & 3_{+}^{1+2}{:}S_3,\, 3^3{:}S_3 & 320 \\ \hline
\end{array}
\]
\caption{Type III factorizations $G_1 = H_1K_1$ with $|H_1|=\ell$ in Lemma  \ref{l:ell3}}
\label{tab4}
\end{table}

\begin{prop}\label{p:ell2}
The conclusion to Theorem \ref{t:main2} holds if $G_0$ is a classical group.
\end{prop}

\begin{proof}
We combine Lemmas \ref{l:ell2} and \ref{l:ell3} with Theorem \ref{t:main}.
\end{proof}

\begin{rem}\label{r:abelian}
There are some natural variations of the parameter $\ell(G_0)$ defined above. For example, let
\[
\mathcal{A}(G_0) = \{ H \,:\, \mbox{$G = HK$, $H$ abelian, $K$ core-free, ${\rm soc}(G) = G_0$}\}
\]
and define $\kappa(G_0) = \min\{ |H| \, :\, H \in \mathcal{A}(G_0)\}$ (set $\kappa(G_0) = 0$ if $\mathcal{A}(G_0)$ is empty). Then one can check that $\mathcal{A}(G_0)$ is nonempty if and only if
\[
G_0 \in \{A_n, {\rm PSL}_{n}(q), {\rm M}_{11}, {\rm M}_{12}, {\rm M}_{23}\}
\]
and we calculate that $\kappa(G_0) = \ell(G_0)$ in each of these cases. We also find that every subgroup in $\mathcal{A}(G_0)$ is cyclic, with the exception of $C_6 \times C_2 \in \mathcal{A}({\rm M}_{12})$.
\end{rem}

\section{Proof of Theorem \ref{t:main3}}\label{s:main3}

In this section we prove Theorem \ref{t:main3} on the exact factorizations of almost simple groups with a solvable factor. We begin with an elementary observation.

\begin{lem}\label{l:easy}
Let $G$ be an almost simple group with socle $G_0$ and suppose $G=HK$ is an exact factorization. Then $|H \cap G_0|$ divides $|G_0: K \cap G_0|$.
\end{lem}

\begin{proof}
We have $G_0H/G_0 \cong H/H \cap G_0$, $G_0K/G_0 \cong K/K \cap G_0$ and $(G_0H/G_0)(G_0K/G_0) = G/G_0$. By taking orders and rearranging, using the fact that $|G|=|H||K|$, we deduce that $\a|H \cap G_0||K \cap G_0| = |G_0|$, where $\a = |(G_0H/G_0) \cap (G_0K/G_0)|$.
\end{proof}

\begin{prop}\label{p:ex1}
The conclusion to Theorem \ref{t:main3} holds if $G_0$ is non-classical.
\end{prop}

\begin{proof}
Here $G_0$ is either a sporadic or alternating group. First assume $G_0$ is sporadic  and $G=HK$ is a factorization. The  possibilities for $G$, $H$ and $K$ are given in \cite[Propositions 4.1 and 4.4]{LX} and it is routine to read off the cases with $|G|=|H||K|$. Similarly, if $G_0 = A_n$ then we apply \cite[Proposition 4.3]{LX}. Here the cases arising in parts (a) and (b) of this proposition are recorded in parts (i) and (ii) of Theorem \ref{t:main3}, and we use {\sc Magma} \cite{magma} to determine the exact factorizations corresponding to the cases that arise in parts (c)-(f) of \cite[Proposition 4.3]{LX}. Note that we can also use {\sc Magma} to read off the exact factorizations coming from the cases recorded in the table in
\cite[Proposition 4.3(b)]{LX}, which allows us to include the condition $H \leqs {\rm A\Gamma L}_{1}(p^a)$ in part (ii) of the theorem.
\end{proof}

\begin{prop}\label{p:ex2}
The conclusion to Theorem \ref{t:main3} holds if $G_0$ is a classical group.
\end{prop}

\begin{proof}
We will consider the Type I, II and III factorizations in turn, as defined in Section \ref{s:intro}. First assume $G=HK$ is a factorization of Type I, so the possibilities for $G$, $H$ and $K$ are described  in Table \ref{tab}. It is easy to see that there are exact factorizations in Case 1 of Table \ref{tab}; for example, we can take $G = {\rm PGL}_{n}(q)$, $H$ a Singer cycle of order $\frac{q^n-1}{q-1}$ and $K=P_1$, noting that $H$ acts transitively on the set of $1$-dimensional subspaces of the natural module for $G$. These cases are recorded in part (iii) of Theorem \ref{t:main3}.

Next let us consider the Type I factorizations in Cases 3 and 4 of Table \ref{tab}. Here $G_0 = {\rm Sp}_{2m}(q)$ with $q$ even and $m \geqs 2$. In Case 3 we have $\O_{2m}^{-}(q) \normeq K \cap G_0$, so $K \cap G_0 = \O_{2m}^{-}(q)$ or ${\rm O}_{2m}^{-}(q)$. By the proof of Proposition \ref{p:case3} we see that $q^m$ divides $|H \cap G_0|$ and by applying Lemma \ref{l:easy} we deduce that $G=HK$ is an exact factorization only if $K \cap G_0 = \O_{2m}^{-}(q)$.

First we claim that if  $m$ is odd then $G = {\rm Sp}_{2m}(q)$ admits an exact factorization $G=HK$ with $H = q^m{:}(q^m-1)$ and $K = \O_{2m}^{-}(q)$. To see this, we construct $H$ as in the proof of Proposition \ref{p:case3} and we note that $H \cap K = X \cap K$, where $X$ is an elementary abelian normal subgroup of $H$ of order $q^m$ with the property that every involution in $X$ is of type $b_m$ (in the notation of \cite{AS}). Since there are no $b$-type involutions in $K$, it follows that $X \cap K=1$ and thus $G=HK$ as claimed.

Next we claim that there are no exact factorizations in Case 3 when $m=2$, which also rules out an exact factorization in Case 4. To see this, we may assume $G=G_0$, $K = \O_{4}^{-}(q)$ and $H \leqs q^3{:}(q^2-1).2<P_2$ has order $q^2(q^2-1)$. Here $H = X{:}C$ is the only possibility, where $X=q^2$ is an irreducible module for a Singer cycle $C<{\rm GL}_{2}(q)$, and by arguing as in the proof of Proposition \ref{p:case3} we deduce that $X = \{1\} \cup z^C$ and $z$ is a $c_2$-type involution in $G$. Moreover, ${\rm fix}(z)\geqs 1$ with respect to the action of $z$ on $\Delta = G/K$ (since $z^G \cap K$ is nonempty), so $H$ does not act regularly on $\Delta$ and we conclude that $G \ne HK$.

It is straightforward to show that none of the other cases in Table \ref{tab} yield an exact factorization. For example, consider Case 2. Here the proof of Proposition \ref{p:case2} implies that $q^3$ divides $|H \cap G_0|$, but we have ${\rm PSp}_{4}(q) \normeq K \cap G_0$ and thus $q^3$ does not divide $|G_0:K \cap G_0|$. Therefore, the existence of an exact factorization in this case is ruled out by Lemma \ref{l:easy}. Similar reasoning applies in Cases 5--9.

Finally, let us turn to the factorizations of Type II and III. For Type III factorizations, the precise possibilities for $G$, $H$ and $K$ are listed in \cite[Table 4.1]{LX} and it is easy to read off the relevant cases. Similarly, the Type II factorizations are presented in \cite[Table 1.2]{LX} and with the aid of {\sc Magma} \cite{magma} and Lemma \ref{l:easy}, it is straightforward to identify the cases that yield an exact factorization.
\end{proof}

\section{Nilpotent regular subgroups}\label{s:main4}

In this final section, we prove Theorem \ref{nil-B-gps}, which describes the primitive permutation groups with a nilpotent regular subgroup. We will also prove Theorem \ref{t:5}. 

Let $G$ be a finite primitive permutation group which contains a nilpotent regular subgroup $H$. Let $T$ denote the socle of $G$. Recall that two groups are \emph{class-mates} if they are isomorphic to regular subgroups of the same primitive permutation group, which is not a symmetric or alternating group in its natural action.

If $G$ is an affine group, then $T$ is an elementary abelian $p$-group for a prime $p$. Moreover, $T$ is regular and thus $H$ is a class-mate of an elementary abelian $p$-group as in part (ii) of Theorem \ref{nil-B-gps}. For the remainder, let us assume $G$ is non-affine, so $T$ is nonabelian. Since $H$ is nilpotent, $H$ does not contain a nontrivial normal subgroup of $T$ and thus the possibilities for $G$ and $H$ are described in \cite[Theorem 1]{LPS-3}. Three cases arise:

\begin{itemize}\addtolength{\itemsep}{0.2\baselineskip}
\item[{\rm (a)}] $G$ is almost simple;
\item[{\rm (b)}] $G$ is a product-type group on $\Delta^d$ of the form $G \leqs G_1 \wr S_d$, where $d \geqs 2$ and $G_1 \leqs {\rm Sym}(\Delta)$ is almost simple and primitive;
\item[{\rm (c)}] $G$ is of diagonal-type (corresponding to cases (ii) and (iii) in \cite[Theorem 1]{LPS-3}).
\end{itemize}

First we handle the diagonal-type groups in case (c). For the proof of the following proposition, we need to recall the well-known Wielandt-Kegel theorem \cite{Kegel,Wielandt-1958}: any finite group which is a product of two nilpotent groups is solvable.

\begin{prop}\label{diag-case}
Let $G$ be a finite primitive permutation group with a nilpotent regular subgroup $H$. Then $G$ is not of diagonal-type.
\end{prop}

\begin{proof}
Seeking a contradiction, first assume that $G$ is `simple diagonal' as in Theorem 1(ii) of \cite{LPS-3}.
Since $H$ is nilpotent, it follows that $T=S^2$ for some nonabelian simple group $S$ and there are proper subgroups $A,B < S$ such that $S=AB$ and $H=A\times B$. But $A$ and $B$ are nilpotent, so the factorization $S=AB$ contradicts the aformentioned Wielandt-Kegel theorem.

Finally, suppose $G$ is `compound diagonal' as in \cite[Theorem 1(iii)]{LPS-3}.
Then $G$ is a blow-up of a simple diagonal primitive group which has a nilpotent regular subgroup. But we have just ruled out the existence of such a simple diagonal group, so this case can also be eliminated.
\end{proof}

To handle cases (a) and (b) above, we need the following simple lemma. Recall that the semi-dihedral and generalized quaternion groups of order $2^{e+1}$ are defined as follows:
\begin{align*}
{\rm SD}_{2^{e+1}} & = \la x,y\mid x^{2^e}=y^2=1, \, x^y=x^{2^{e-1}-1}\ra \\
Q_{2^{e+1}} & =\la x,y\mid x^{2^e}=y^4=1, \, x^y=x^{-1}, \, x^{2^{e-1}}=y^2\ra
\end{align*}

\begin{lem}\label{metac-nil}
Let $A$ be a nilpotent subgroup of ${\rm GL}_{1}(q^m){:}m \leqs \GammaL_1(q^m)$ with  $m \geqs 2$. Then $A$ is transitive on the set of nonzero vectors of $\mathbb{F}_q^m$
if and only if
\begin{itemize}\addtolength{\itemsep}{0.2\baselineskip}
\item[{\rm (i)}] $A = {\rm GL}_{1}(q^m) = C_{q^m-1}$; or 
\item[{\rm (ii)}] $m=2$, $q$ is a Mersenne prime and $A=C_{\frac{1}{2}(q-1)}\times B$ with $B = {\rm SD}_{4(q+1)}$ or $Q_{2(q+1)}$.
\end{itemize}
\end{lem}

\begin{proof}
Let $\O$ be the set of nonzero vectors of $\mathbb{F}_q^m$.
Clearly, if (i) holds then $A$ is transitive on $\O$. Now assume (ii) holds. If $B = {\rm SD}_{4(q+1)}$ then $A = \GammaL_1(q^2)$ is transitive. If $B = Q_{2(q+1)}$ then both $A$ and $B$ have a unique involution, which lies in ${\rm GL}_{1}(q^2)$ and fixes no nonzero vector of $\mathbb{F}_q^2$. It follows that $A$ acts regularly on $\O$. 

For the remainder, we may assume that $A$ acts transitively on $\O$. In particular, $|A|$ is divisible by $q^m-1$. 

First assume that $q^m-1$ has a primitive prime divisor $r$ and let $a\in A$ be of order $r$. Let $R$ be the unique Sylow $r$-subgroup of $A$ and note that $R$ is cyclic since $r>m$. Since $A$ is nilpotent, we have $A=R\times S$, where $S$ is a Hall $r'$-subgroup of $A$, and thus $a \in Z(A)$.  Therefore,
\[
A = C_A(a) \leqs C_{{\rm GL}_{1}(q^m){:}m}(a) = {\rm GL}_{1}(q^m)
\]
and thus (i) holds.

To complete the proof, let us now assume that $q^m-1$ does not have a primitive prime divisor.
Then either $q^m=2^6$, or $m=2$ and $q$ is a Mersenne prime. In the former case, it is straightforward to check that (i) holds, so assume $m=2$ and $q$ is a Mersenne prime. Write ${\rm GL}_{1}(q^2).2 = \GammaL_1(q^2) = \la x\ra{:}\la y\ra = C_{q^2-1}{:}C_2$ with $x^y=x^q$.
Since $q+1$ is a power of 2, we have $q^2-1=2(q+1)\cdot \frac{1}{2}(q-1)$ and $\frac{1}{2}(q-1)$ is odd. Let $x_1=x^{2(q+1)}$ and $x_2=x^{\frac{1}{2}(q-1)}$.
Then $|x_1|=\frac{1}{2}(q-1)$, $|x_2|=2(q+1)$ and $\la x\ra=\la x_1\ra\times\la x_2\ra$.
Now $x_1^y=x_1^q=x_1$ and $x_2^y=x_2^q$, so
\[
\GammaL_1(q^2)=\la x_1,x_2,y\ra=\la x_1\ra\times\la x_2,y\ra=C_{\frac{1}{2}(q-1)}\times {\rm SD}_{4(q+1)}
\]
is nilpotent. Finally, set $z=x_2y$ and note that $z^2=x_2yx_2y=x_2^{q+1}$. Since $|x_2|=2(q+1)$, it follows that $z$ has order $4$ and we deduce that
\[
\la x_1\ra\times\la x_2^2, x_2y\ra=C_{\frac{1}{2}(q-1)}\times Q_{2(q+1)}.
\]
This completes the proof.
\end{proof}

In order to determine the almost simple candidates from Theorem~\ref{t:main3}, we need the following projective version of Lemma~\ref{metac-nil}. In the statement, we write $\frac{q^m-1}{q-1}{:}mf$ for the image of ${\rm GL}_{1}(q^m){:}mf < {\rm \Gamma L}_{m}(q)$ modulo scalars, where $q=p^f$ with $p$ prime.

\begin{lem}\label{metac-nil'}
Let $A$ be a nilpotent subgroup of $\frac{q^m-1}{q-1}{:}mf$ with $m \geqs 2$. Then $A$ is transitive on the set of $1$-dimensional subspaces of $\mathbb{F}_q^m$ if and only if one of the following holds:
\begin{itemize}\addtolength{\itemsep}{0.2\baselineskip}
\item[{\rm (i)}] $A$ is a cyclic group of order $\frac{q^m-1}{q-1}$.
\item[{\rm (ii)}] $m=2$, $q$ is a Mersenne prime and $A=D_{2(q+1)}$ or $D_{q+1}$.
\item[{\rm (iii)}] $m=2$, $q=8$ and $A = 3_{-}^{1+2}$.
\end{itemize}
In (ii), we note that $\frac{q^m-1}{q-1}{:}mf = D_{2(q+1)}$ has two subgroups isomorphic to $D_{q+1}$; one is transitive, the other is intransitive.
\end{lem}

\begin{proof}
This is very similar to the proof of the previous lemma. Suppose $A$ is transitive, so $|A|$ is divisible by $\frac{q^m-1}{q-1}$. If there exists a primitive prime divisor $r$ of $p^{mf}-1$, then by considering the centralizer of an element in $A$ of order $r$ we deduce that (i) holds. Now assume $p^{mf}-1$ does not have a primitive prime divisor,  so either $p^{mf} = 2^6$, or $m=2$ and $q$ is a Mersenne prime. In the former case, $q \in \{2,4,8\}$ and it is straightforward to check that either (i) holds, or $q=8$ and $A < C_9{:}C_6 < {\rm P\Gamma L}_{2}(8)$ is an extraspecial group of type $3_{-}^{1+2}$ as in (iii). Finally, if $m=2$ and $q$ is a Mersenne prime, then $\frac{q^m-1}{q-1}{:}mf = D_{2(q+1)}$ and we deduce that (i) or (ii) holds. Specifically, if $D_{2(q+1)} = \la x,y \ra$, where $|x|=q+1$ and $y$ is an involution fixing a $1$-space, then $D_{q+1}=\la x^2,xy\ra$ is transitive and $D_{q+1} = \la x^2,y\ra$ is intransitive.
\end{proof}

Next we establish a key result for the analysis of product-type groups. 

\begin{prop}\label{nil--abel}
Let $G$ be an almost simple primitive group of degree $n$ with socle $G_0$ and suppose $H$ is a nilpotent transitive subgroup of $G$. Then one of the following holds:
\begin{itemize}\addtolength{\itemsep}{0.2\baselineskip}
\item[{\rm (i)}] $G$ is $2$-transitive and $H=C_n$ is regular.
\item[{\rm (ii)}] $G_0=A_n$ and $H$ is transitive on $\{1, \ldots, n\}$.
\item[{\rm(iii)}] $G_0={\rm PSL}_2(q)$, $n=q+1$ and $H=D_{2(q+1)}$ or $D_{q+1}$, where $q$ is a Mersenne prime.
\item[{\rm (iv)}] $(G,H,n)=({\rm P\Gamma L}_{2}(8),3_{-}^{1+2},9)$, $({\rm M}_{12}, C_6\times C_2,12)$ or $({\rm M}_{24},D_8 \times C_3,24)$.
\item[{\rm (v)}] $G_0 = {\rm PSp}_{4}(3)$, $n=27$ and $H=3_{\pm}^{1+2}$ or $3.3_{+}^{1+2}$.
\end{itemize}
In particular, $G$ is $2$-transitive in cases {\rm (i)}--{\rm (iv)}.
\end{prop}

\begin{proof}
Since we have a factorization $G=HK$, where $K$ is a point stabilizer, we can proceed by inspecting the possibilities for $(G,H)$ determined in \cite{LX}. Set $\Delta = G/K$.

Set $H_0 = H \cap G_0$ and first assume $G_0$ is a classical group over $\mathbb{F}_q$, where $q=p^f$ with $p$ a prime. We consider Type I, II and III factorizations in turn (see Section \ref{s:intro} for the definitions).
To begin with, let us assume $G=HK$ is a Type I factorization, in which case the possibilities for $G$ and $H$ are described in Table~\ref{tab}. 
Let $A \leqs G_0$ be the overgroup of $H_0$ given in Table~\ref{tab}.

First consider Case~1 of Table~1. Here $G_0={\rm PSL}_m(q)$ and $H \leqs \frac{q^m-1}{q-1}{:}mf$ acts transitively on $\Delta$, which we identify with the 
set of $1$-dimensional subspaces of $\mathbb{F}_q^m$. Now $G$ acts $2$-transitively on $\Delta$ and by applying Lemma~\ref{metac-nil'}, we deduce that either (i) or (iii) holds, or we are in the special case $(G,H,n)=({\rm P\Gamma L}_{2}(8),3_{-}^{1+2},9)$ recorded in (iv).

Now consider Case 2 of Table~1, so $G_0 = {\rm PSL}_4(q)$ and $A=q^3{:}{\frac{q^3-1}{d}}{:}3$ with $d=(4,q-1)$. By Theorem \ref{t:main}, we know that $|H|$ is divisible by $q^3(q^3-1)/(2,q-1)$. Now
\[
H<P_1\leqs \AGammaL_3(q)=q^3{:}(\GL_3(q).f)
\]
and $O_p(H)=O_p(P_1)$ is the full unipotent radical of the parabolic subgroup $P_1$ (here $O_p(H)$ denotes the largest normal $p$-subgroup of $H$).
It follows that $H \leqs \AGammaL_1(q^3)=q^3{:}\GammaL_1(q^3)$. But $\AGL_1(q^3)=q^3{:}C_{q^3-1}$ is a Frobenius group, and so is $H\cap\AGL_1(q^3)$, whence $H$ is not nilpotent and we have reached a contradiction.

Next let us turn to Case 3 of Table~1, so $G_0 ={\rm Sp}_{2m}(q)$ and $A = q^{\frac{1}{2}m(m+1)}{:}(q^m-1).m$ with $q$ even. By Theorem \ref{t:main}, $|H|$ is divisible by $q^m(q^m-1)$. Let $L=H/O_2(H)$ and note that $O_2(H)$ is elementary abelian. In addition, observe that $K \cap G_0 = {\rm O}_{2m}^{-}(q)$ since $K$ is a maximal subgroup of $G$.
Then
\[
L\leqs {\rm \Gamma L}_{1}(q^m) = {\rm GL}_{1}(2^{mf}).mf
\]
and since $q=2^f$ is even, we deduce that $L=C_{q^m-1}$ is cyclic by Lemma~\ref{metac-nil}. Since $H$ is nilpotent, it follows that $H=O_2(H)\times C_{q^m-1}$ is abelian and therefore regular on $\Delta$. But this is incompatible with Theorem \ref{t:main3} (specifically, the condition $K \cap G_0 = \O_{2m}^{-}(q)$ in part (iv) of Theorem \ref{t:main3}). Similar reasoning eliminates all the other  cases in Table~\ref{tab}.

Next assume $G=HK$ is a Type II factorization. Here the possibilities for $(G,H)$ are listed in \cite[Table 1.2]{LX}, which gives an overgroup $A \leqs G_0$ of $H_0$. In addition, by Lemma \ref{l:ell2}, $|H|$ is divisible by the integer $\ell$ in Table \ref{tab3}. With the exception of Case 1 (which is $2$-transitive) and Cases 10 and 11 (which give rise to the examples recorded in part (v) of the proposition), it is straightforward to check that $H$ does not contain a nilpotent subgroup of order divisible by $\ell$. Similarly, if $G=HK$ is a Type III factorization, then the desired result follows immediately from \cite[Proposition 4.1]{LX} (note that $G$ is $2$-transitive in the cases labelled 0, 1 and 4 in Table \ref{tab4}).

To complete the proof, we may assume $G_0$ is a sporadic or alternating group (recall that no almost simple exceptional group of Lie type admits a solvable factor). For $G_0$ sporadic, we apply \cite[Proposition 4.4]{LX}, which gives rise to the two examples appearing in part~(iv). Finally, let us assume $G_0 = A_n$ is an alternating group. Here we apply \cite[Proposition 4.3]{LX}, which describes the possibilities for $(G,H)$. We will refer to cases (a)--(f) of
\cite[Proposition 4.3]{LX}. Clearly, (ii) holds in case (a), and $H$ is not nilpotent in (c). Cases (d) and (e) can be checked directly; there is an example with $G = S_6$, $H=C_6$ and $K={\rm PGL}_{2}(5)$, but here the action of $G$ on $\Delta$ is $2$-transitive and (i) holds. Similarly, in (f) the only valid possibility is $G = A_8$, $H=C_{15}$ and $K = {\rm AGL}_{3}(2)$, and once again we find that $G$ is $2$-transitive. Finally, let us consider part (b). Here $G = S_n$ or $A_n$, $K = (S_{n-2} \times S_2) \cap G$ and $n=p^a$ is a prime power. Moreover, the nilpotency of $H$ implies that $H \leqs {\rm A\Gamma L}_{1}(p^a) = p^a{:}({\rm GL}_{1}(p^a).a)$, 
and $|H|$ must be divisible by $|G:K| = \frac{1}{2}p^a(p^a-1)$.
It follows that $O_p(H)=p^a$ and $H/O_p(H)\leqs\GL_1(p^a).a$ contains a nontrivial element of $\GL_1(p^a)$, which does not centralize $O_p(H)$. This contradicts the nilpotency of $H$ and the proof is complete.
\end{proof}

By combining the previous result (and its proof) with Theorem \ref{t:main3}, it is straightforward to read off the almost simple primitive groups with a nilpotent regular subgroup.

\begin{cor}\label{c:nil--abel}
Let $G$ be an almost simple primitive group of degree $n$ with socle $G_0$ and suppose $H$ is a nilpotent regular subgroup of $G$. Then one of the following holds:
\begin{itemize}\addtolength{\itemsep}{0.2\baselineskip}
\item[{\rm (i)}] $G_0=A_n$ and $H$ is transitive on $\{1, \ldots, n\}$.
\item[{\rm (ii)}] $G_0 = {\rm PSL}_{a}(q)$ and $H$ is a cyclic group of order 
$\frac{q^a-1}{q-1}$.
\item[{\rm(iii)}] $G_0={\rm PSL}_2(q)$ and $H=D_{q+1}$, where $q$ is a Mersenne prime.
\item[{\rm (iv)}] $(G,H)$ is one of the following:
\[
\mbox{$({\rm PSL}_{2}(11),C_{11})$, $({\rm M}_{11},C_{11})$, $({\rm M}_{12}, C_6\times C_2)$, $({\rm M}_{23},C_{23})$ or $({\rm M}_{24},D_8 \times C_3)$.}
\]
\item[{\rm (v)}] $G_0 = {\rm PSp}_{4}(3)$ and $H=3_{\pm}^{1+2}$.
\end{itemize}
\end{cor}

\begin{prop}\label{pa-case}
Let $G \leqs G_1 \wr S_d$ be a primitive product-type group, where $G_1$ is almost simple and primitive. If $G$ has a nilpotent regular subgroup, then either $G_1$ is $2$-transitive, or $G_1$ has socle $\PSp_4(3)$ acting on $27$ points.
\end{prop}

\begin{proof}
By \cite[Proposition 2.5]{LPS-3}, $G_1$ has a nilpotent transitive subgroup. Then Proposition \ref{nil--abel} implies that either $G_1$ is $2$-transitive, or $G_1$ has socle ${\rm PSp}_{4}(3)$ and a regular subgroup $3_{+}^{1+2}$. The result follows.
\end{proof}

We have already noted that  $(3_{\pm}^{1+2})^d$ is a regular subgroup of ${\rm PSp}_4(3)\wr S_d$ in its product action of degree $27^d$. However, ${\rm PSp}_4(3)\wr S_d$ is \emph{not} the blow-up of a $2$-transitive group (since the action of ${\rm PSp}_{4}(3)$ on the set of cosets of $2^4{:}A_5$ is simply primitive) and thus $(3_{\pm}^{1+2})^d$ is \emph{not} a generalized B-group.

\begin{example}\label{3^2:3}
We claim that the primitive group ${\rm PSp}_4(3)\wr S_d$ in its product action of degree $27^d$ has a regular subgroup isomorphic to $(3_+^{1+2})^e\times(3^6{:}3_+^{1+2})^f$ for any non-negative integers $e$ and $f$ with $e+3f=d$.

To see this, we may assume that $e=0$ and $f=1$, so $d=3$. let $P_i=X_i{:}Y_i\cong3_+^{1+2}$, where $1\leqslant i\leqslant 3$, $X_i\cong C_3^2$ and $Y_i=\la y_i\ra\cong C_3$.
Let
\[
(P_1\times P_2\times P_3){:}\la \pi\ra\cong 3_+^{1+2}\wr C_3,
\]
where $\pi: y_1\mapsto y_2$, $y_2\mapsto y_3$ and $y_3\mapsto y_1$, and set
\[
H=(X_1\times X_2\times X_3){:}(\la y_1y_2^{-1},y_2y_3^{-1}\ra{:}\la y_1y_2y_3\pi\ra).
\]
Then $\la y_1y_2^{-1},y_2y_3^{-1}\ra{:}\la y_1y_2y_3\pi\ra\cong 3_+^{1+2}$ and
$H \cong 3^6{:}3_+^{1+2}$.
Suppose that $y_1y_2y_3\pi$ fixes a point $x_1x_2x_3$.
Then
\[
x_1x_2x_3=(x_1x_2x_3)^{y_1y_2y_3\pi}=(x_1y_1x_2y_2x_3y_3)^\pi=x_2y_2x_3y_3x_1y_1
\]
and thus $y_1y_2y_3=1$, which is not possible. It follows that $H$ is regular and this  justifies the claim.

We remark that ${\rm PSp}_4(3)\wr S_d$ appears to have many regular subgroups in its product action of degree $27^d$. For example, if $d=2$ then a {\sc Magma} calculation shows that there are $19$ conjugacy classes of regular subgroups ($13$ up to isomorphism).
\end{example}

Finally, we are ready to prove Theorems~\ref{nil-B-gps} and \ref{t:5}.

\begin{proof}[Proof of Theorem~\ref{nil-B-gps}]
Let $G$ be a finite primitive permutation group of degree $n$ with socle $T$ and a nilpotent regular subgroup $H$. Then by the observations made at the beginning of this section, together with Propositions \ref{diag-case} and \ref{pa-case}, we deduce that one of the following holds:
\begin{itemize}\addtolength{\itemsep}{0.2\baselineskip}
\item[{\rm(a)}] $G$ is affine; 
\item[{\rm(b)}] $G \leqs G_1 \wr S_d$ is of product-type, where $G_1$ is a primitive group with socle $\PSp_4(3)$ acting on $27$ points;
\item[{\rm(c)}] $G$ is a blow-up of an almost simple $2$-transitive group.
\end{itemize}

If $G$ is an affine group of degree $p^d$ with $p$ prime and $d\geqslant1$, then $H$ is a class-mate of $C_p^d$ as in part (ii) of Theorem~\ref{nil-B-gps}. Similarly, if (b) holds then $H$ is a class-mate of $(3_+^{1+2})^d$ as in part (iv) of Theorem~\ref{nil-B-gps}. To complete the proof, we may assume that (c) holds. 

Let us first assume $G$ is almost simple. By applying Corollary~\ref{c:nil--abel}, either $T=A_n$ and $H$ is any nilpotent group of order $n$ as in part (i) of Theorem~\ref{nil-B-gps}, or $G$ is one of the groups described in part (iii) of Theorem~\ref{nil-B-gps} (with $d=1$). 

Finally, suppose $G \leqs G_1 \wr S_d$ is a blow-up of an almost simple $2$-transitive group $G_1$ of degree $m$, so $n=m^d$ with $d \geqs 2$. Write $T=S^d$, where $S$ is the socle of $G_1$. As observed in the proof of Proposition \ref{pa-case}, $G_1$ has a nilpotent transitive subgroup $H_1$ and therefore the possibilities for $G_1$ and $H_1$ are described in Proposition~\ref{nil--abel}. We inspect each case in turn.

First assume that $H_1=C_m$ as in part (i) of Proposition~\ref{nil--abel}. Then $H_1$ is regular and thus $G_1$ is one of the groups in parts (i), (ii) or (iv) of Corollary~\ref{c:nil--abel}. We deduce that $H$ is a class-mate of $C_m^d$ and either $S=A_m$, or $(S,m) = ({\rm PSL}_{a}(q),\frac{q^a-1}{q-1})$, $({\rm PSL}_{2}(11),11)$, $({\rm M}_{11},11)$ or $({\rm M}_{23},23)$. All of these cases appear in part (iii)(a) of Theorem~\ref{nil-B-gps}.

Next suppose $G_1$ and $H_1$ are as in part (ii) of Proposition~\ref{nil--abel}. Here $S=A_m$ and $H$ is a class-mate of $C_m^d$, which we have already noted is one of the cases in part (iii)(a) of Theorem~\ref{nil-B-gps}. Similarly, if part (iii) of Proposition~\ref{nil--abel} is satisfied, then $S = {\rm PSL}_{2}(q)$, $m=q+1$ and $H$ is a class-mate of $C_m^d$. Finally, if $G_1$ and $H_1$ are as in part (iv) of Proposition~\ref{nil--abel} then either $S = {\rm PSL}_{2}(8)$ and $H$ is a class-mate of $C_9^d$, or $S = {\rm M}_{12}$ or ${\rm M}_{24}$ and $H$ is a class-mate of $(C_6 \times C_2)^d$ or $(D_8 \times C_3)^d$, respectively. All of these cases are covered by part (iii) of Theorem~\ref{nil-B-gps}. This completes the proof.
\end{proof}

\begin{proof}[Proof of Theorem~\ref{t:5}]
Let $H$ be a finite nilpotent group and let $G$ be a primitive permutation group containing $H$ as a regular subgroup. Then as explained in the first paragraph in the proof of Theorem~\ref{nil-B-gps}, $G$ is either affine (and thus $H$ is a class-mate of $C_p^d$ for some prime $p$ and $d \geqs 1$), or $G$ is a blow-up of an almost simple primitive group with socle ${\rm PSp}_{4}(3)$ acting on $27$ points (in which case $H$ is a class-mate of $(3_{+}^{1+2})^d$), or $G$ is the blow-up of a $2$-transitive group. The desired conclusion follows immediately.
\end{proof}

\end{document}